\documentclass[11pt,reqno]{amsart}
\usepackage{amsbsy,amsfonts,amsmath,amssymb,amscd,amsthm}
\usepackage{mathrsfs,float,color}
\usepackage[mathcal]{euscript}
\usepackage{subfigure,enumitem,graphicx,epstopdf}
\usepackage[a4paper,text={6.1in,8in},centering,top=1.5in]{geometry}
\usepackage{pxfonts,epstopdf,multicol}
\usepackage[colorlinks=true,linkcolor=blue,citecolor=green]{hyperref}
\usepackage{srcltx}
\usepackage[margin=20pt,font=small,labelfont=bf,textfont=it]{caption}
\usepackage[latin1]{inputenc}
\usepackage{mathtools}

\hypersetup{linktocpage}

\small\normalsize
\theoremstyle{plain}
\newtheorem{mainthm}{Theorem}
\newtheorem{maincor}[mainthm]{Corollary}

\newtheorem{thm}{Theorem}[section]
\newtheorem{cor}[thm]{Corollary}
\newtheorem{lem}[thm]{Lemma}

\newtheorem{prop}[thm]{Proposition}

\newtheorem{defi}[thm]{Definition}

\theoremstyle{definition}

\newtheorem{rem}[thm]{Remark}

\newcommand{\eqdef}{\stackrel{\scriptscriptstyle\rm def}{=}}
 \DeclareMathOperator{\diam}{diam}

\makeatletter
\def\l@part{\@tocline{0}{-2pt}{1pc}{}{}}
\def\l@section{\@tocline{1}{-2pt}{1pc}{4.6em}{}}
\renewcommand{\tocpart}[3]{%
  \indentlabel{\@ifnotempty{#2}{\makebox[2.3em][l]{%
    \ignorespaces#1 #2.\hfill}}}\bf{#3}}
\renewcommand{\tocsection}[3]{%
  \indentlabel{\@ifnotempty{#2}{\hspace*{2.3em}\makebox[2.3em][l]{%
    \ignorespaces#1 #2.\hfill}}}#3}

\makeatother

\newcommand{\supp}{\operatorname*{supp}}

\let\oldtocsection=\tocsection
\let\oldtocsubsection=\tocsubsection
\renewcommand{\tocsection}[2]{\hspace{0em}\bf\oldtocsection{#1}{#2}}
\renewcommand{\tocsubsection}[2]{\hspace{4.8em}\oldtocsubsection{#1}{#2}}
\let\oldtocsubsubsection=\tocsubsubsection
\renewcommand{\tocsubsubsection}[2]{\hspace{4.2em}\oldtocsubsubsection{#1}{#2}}

\begin{document}
~\vspace{-2.0cm} 
\title[Markovian random maps]{Finitude of physical measures for Markovian random maps}

\author[Barrientos]{Pablo G.~Barrientos}
\address[Pablo G.~Barrientos]{Instituto de Matem\'atica e Estat{\'{i}}stica, Universidade de Federal Fluminense, Niter\'{o}i, Brazil}
\email{pgbarrientos@id.uff.br}

\author[D. Malicet]{Dominique Malicet}
\address[Dominique Malicet]{Universit\'e Gustave Eiffel, Laboratoire d'analyse et de math\'ematiques appliqu\'ees (LAMA), 5 Bd Descartes, 77420 Champs-sur-Marne, Francia}
\email{dominique.malicet@univ-eiffel.fr}

\author[Nakamura]{Fumihiko Nakamura}
\address[Fumihiko Nakamura]{Faculty of Engineering, Kitami Institute of Technology, Hokkaido, 090-8507, JAPAN}
\email{nfumihiko@mail.kitami-it.ac.jp}

\author[Nakano]{Yushi Nakano}
\address[Yushi Nakano]{Faculty of Science, Hokkaido University, Hokkaido, 060-0810, Japan}
\email{yushi.nakano@math.sci.hokudai.ac.jp}

\author[Toyokawa]{Hisayoshi Toyokawa}
\address[Hisayoshi Toyokawa]{Faculty of Engineering, Kitami Institute of Technology, Hokkaido, 090-8507, JAPAN}
\email{h\_toyokawa@mail.kitami-it.ac.jp.}

\subjclass{Primary 37A30, 37C40, 37H05; Secondary 37A50, 37C30, 60J05}

\keywords{Markovian random maps, finitude of physical measures, Markov operators, constrictiveness, mostly contracting random maps}

\begin{abstract}

We study the finiteness of physical measures for skew-product transformations $F$ associated with discrete-time random dynamical systems driven by ergodic Markov chains. We develop a framework, using an independent and identically distributed (i.i.d.) representation of the Markov process, that facilitates transferring results from the well-studied Bernoulli (i.i.d.) setting to the Markovian context.

Specifically, we establish conditions for the existence of finitely many ergodic, $F$-invariant measures, absolutely continuous with respect to a reference measure, such that their statistical basins of attraction for measurable bounded observables cover the phase space almost everywhere. Furthermore, we investigate a weaker notion, which demands finitely many physical measures (not necessarily absolutely continuous) whose weak$^*$ basins of attraction cover the phase space almost everywhere. We show that for random maps on compact metric spaces driven by Markov chains on finite state spaces, this property holds if the system is mostly contracting, i.e., if all the Markovian invariant measures have negative maximal Lyapunov exponents. This result is applied to random $C^1$ diffeomorphisms of the circle and the interval under conditions based on the absence of invariant probability measures or finite invariant sets, respectively. We also connect our result to the quasi-compactness of the Koopman operator on the space of H\"older continuous functions. 
\end{abstract}

\setcounter{tocdepth}{2} \maketitle
\thispagestyle{empty}

\section{Introduction}

A discrete random dynamical system is defined by a measurable map  \( f:T \times X \to X \) referred to as a \emph{random map} where \((T,\mathscr{A})\) and \((X,\mathscr{B})\) are measurable spaces called the \emph{state space} and \emph{phase space}, respectively. From this map, one considers the non-autonomous iterations:
\[
    f^0_\omega = \mathrm{id}, \quad \text{and} \quad f^n_\omega = f_{\omega_{n-1}} \circ \dots \circ f_{\omega_0}, \quad n \geq 1, \quad \omega = (\omega_i)_{i\geq 0} \in \Omega,
\]
where \( (\Omega,\mathscr{F}) = (T^{\mathbb{N}}, \mathscr{A}^{\mathbb{N}}) \) is called the \emph{sample space} (of all possible noise realizations) and \( f_t := f(t,\cdot) \) for \( t \in T \). In other words, this defines a \emph{locally constant cocycle} over the shift transformation $\sigma$ on \( \Omega \) which can be seen as the fiber dynamics of the  skew-product
\begin{equation*}\label{eq:skew}
F: \Omega\times X \to \Omega\times X, \quad F(\omega,x)=(\sigma(\omega),f_{\omega_0}(x)).
\end{equation*}

However, we are not interested in studying the evolution of the dynamics for all possible compositions. Instead, we focus on the sequence of compositions determined by a generic sequence \( \omega \in \Omega \), according to a given ergodic shift-invariant probability measure \( \mathbb{P} \) on \( \Omega \), which we refer to as \emph{noise}. A common choice for \( \mathbb{P} \) is an independent and identically distributed (i.i.d.) measure, also known as a \emph{Bernoulli measure}, which serves as a starting point for analyzing the dynamics of a random map. A natural next step is to extend the results to \emph{Markov measures} on \( \Omega \).

One of the main objectives of this paper is to develop a framework that facilitates the straightforward transfer of results from Bernoulli noise to Markov noise. We apply this approach to demonstrate that some of the results obtained in recent works~\cite{BNNT22} and~\cite{BM24} concerning the Palis conjecture in the i.i.d.~random setting can be extended to Markov noise with relative success. 
Before presenting this framework and stating our results, we introduce the problem setting in more detail.

\subsection{Markov measures}  \label{ss:Markov}

Consider a \emph{transition probability} \( Q(t,A) \) on \( T \), defined as a function \( Q: T \times \mathscr{A} \to [0,1] \) such that \( Q(t,\cdot) \) is a probability measure for all \( t \in T \), and \( Q(\cdot,A) \) is a measurable function for all \( A \in \mathscr{A} \). Let \( p \) be a \( Q \)-stationary measure on \( T \), meaning that \( p = \int Q(t,\cdot) \ dp(t) \). Then, there is a unique shift-invariant probability measure \( \mathbb{P} \) on \( \Omega  = T^\mathbb{N} \) satisfying
\[
    \mathbb{P}(A_0 \times \dots \times A_{m-1} \times \Omega) = \int_{A_0} \dots \int_{A_{m-1}} Q(t_{m-2}, dt_{m-1}) \dots Q(t_0, dt_1)\, dp(t_0)
\]
for every \( m \in \mathbb{N} \) and all sets \( A_0, \dots, A_{m-1} \in \mathscr{A} \), cf.~\cite{Revuz84}. The measure \( \mathbb{P} \) is called the \emph{Markov measure} with transition probability \( Q(t,A) \) and initial distribution \( p \). We observe that if the transition probability is given by $Q(t,A)=p(A)$ for all $t\in T$, then the Markov measure $\mathbb{P}$ with this transition probability and initial distribution $p$ is the \emph{Bernoulli measure} $\mathbb{P}=p^\mathbb{N}$. Thus, Markov measures extend Bernoulli measures.

Following \cite{crauel1991markov,matias:2022}, we introduce the following definition, which will be useful throughout the paper.

\begin{defi}\label{def:markovian measure}
An \(F\)-invariant probability measure \(\bar{\mu}\) on \(\Omega \times X\) is said to be \emph{\(\mathbb{P}\)-Markovian} if its first marginal is \(\mathbb{P}\) and if it admits a disintegration of the form
$d\bar{\mu} = \bar{\mu}_{\omega_0}\,d\mathbb{P}(\omega)$.
\end{defi}

In the Bernoulli case, i.e., when \(\mathbb{P}=p^\mathbb{N}\), the \(F\)-invariance of a \(\mathbb{P}\)-Markovian measure \(\bar{\mu}\) on \(\Omega \times X\) forces the fiber measures \(\bar{\mu}_\omega\) in its disintegration to be almost surely equal to a constant measure \(\mu\); hence, \(\bar{\mu}=\mathbb{P}\times \mu\). See Lemma~\ref{lem:Markov-Bernoulli} for details. The measure \(\mu\) on \(X\) is called \emph{\(f\)-stationary} and is characterized by the formula
$\mu = \int f_t \mu \,dp(t)$,  
cf.~\cite{O83,kifer-ergodic}.

\subsection{Finitude of Physical Measures} 
Let \((T, \mathscr{A}, p)\) be a probability space and fix a probability measure \(m\),  called the \emph{reference measure}, on the Borel $\sigma$-field $\mathscr{B}$ of a Polish space $X$.  

The finitude of ergodic physical measures is a central problem in dynamical systems. Palis~\cite{Palis2000} conjectured that for most (smooth) dynamical systems, there exist finitely many physical measures that capture the statistical behavior of almost every orbit. Here, physicality means that the basin of attraction of each measure has a positive reference measure and that these measures describe the asymptotic statistics in the sense that their basins cover the whole phase space modulo a null set with respect to \(m\).

\enlargethispage{1cm}
In~\cite{BNNT22}, the authors characterize the finitude of physical measures in the context of nonsingular random maps driven by Bernoulli noise, that is, when \(\mathbb{P}=p^\mathbb{N}\). It was introduced the property 
\emph{{\tt (FPM)} with respect to \(m\) for a Bernoulli random map \(f\)} as the existence of finitely many ergodic \(f\)-stationary probability measures \(\mu_1,\dots,\mu_r\) on \(X\) satisfying:
\begin{enumerate}[label=\arabic*)]
 \item Each \(\mu_i\) is absolutely continuous with respect to \(m\);
 \item Their supports are pairwise disjoint (up to an \(m\)-null set);\footnote{Let $(Y,\mathscr{G},\eta)$ be a probability space. The support of a probability measure $\mu$ (up to $\eta$-null set or relative to $\eta$) is any set $S \in  \mathscr{G}$ such that $\mu( Y \setminus S  ) = 0$  and for any $A \in \mathscr{G}$ with $A \subset S$ and $\mu( A ) = 0$, we have $\eta( A ) = 0$. If $d\mu=  h \, d\eta$, this support coincides with the support of density $h$ up to an $\eta$-null set.}
 \item For any measurable bounded function \(\psi: X \to \mathbb{R}\),
  \begin{equation*}\label{eq:10100a}
  m\Bigl(B_\omega(\mu_1, \psi)\cup \dots \cup B_\omega(\mu_r, \psi)\Bigr)=1 \quad \text{for \(\mathbb{P}\)-almost every \(\omega \in \Omega\),}
  \end{equation*}
where 
\begin{equation*}\label{eq:10100b}
\quad B_\omega(\mu_i, \psi)=\bigg\{ x\in X : \, \lim_{n\to \infty} \frac{1}{n} \sum_{j=0}^{n-1} \psi\bigl(f_{\omega}^{j}(x)\bigr) = \int \psi \, d\mu_i \bigg\} \quad \text{for \(i=1,\dots,r\).}
\end{equation*}
\end{enumerate}

We now revisit the above definition of {\tt (FPM)} as follows:

\begin{defi}\label{dfn:11}
Let \(\mathbb{P}\) be an ergodic shift-invariant Markov measure on \(\Omega\). We say that the skew-product \(F\) associated with  \(f\) satisfies \emph{{\tt (FPM)} with respect to \(\bar{m}=\mathbb{P}\times m\)} if there is finitely many ergodic \(F\)-invariant probability measures \(\bar{\mu}_1,\dots,\bar{\mu}_r\) on \(\Omega \times X\) such that for each \(i=1,\ldots,r\),
\begin{enumerate}[label=\arabic*),start=1]
   \item \(\bar{\mu}_i\) is a \(\mathbb{P}\)-Markovian measure, i.e., \(d\bar{\mu}_i = \mu_{i,\omega_0}\,d\mathbb{P}(\omega)\);
   \item \(\bar{\mu}_i \ll \bar{m}\), i.e., \(\bar{\mu}_i\) is absolutely continuous with respect to \(\bar{m}\);
   \item For any measurable bounded function \(\varphi:\Omega \times X \to \mathbb{R}\),
   \[
   \bar{m}\big(B(\bar{\mu}_1, \varphi) \cup \dots \cup B(\bar{\mu}_r, \varphi)\big)=1,
   \]
   where 
   \[
   B(\bar{\mu}_i, \varphi) = \bigg\{ (\omega, x) \in \Omega \times X : \, \lim_{n \to \infty} \frac{1}{n} \sum_{j=0}^{n-1} \varphi\bigl(F^j(\omega, x)\bigr) = \int \varphi \, d\bar{\mu}_i \bigg\}.
   \]
\end{enumerate}
\end{defi}

\begin{rem}[Basin of attraction] \label{rem:basin}
Assume that \(\Omega\) is a Polish space. Then \(\Omega \times X\) is also Polish and, hence, admits a complete metric \(d\). In this setting, the convergence in the weak* topology on the space of probability measures on \(\Omega \times X\) is determined by a countable set \(S\) of bounded Lipschitz functions (with respect to \(d\)), cf.~\cite[Proposition~5.1]{BNNT22}. Consequently, the \emph{statistical basin of attraction} \(B(\bar{\mu}_i)\) of \(\bar{\mu}_i\) for \(F\) can be expressed as
\[
B(\bar{\mu}_i) \coloneqq \bigg\{ (\omega,x)\in \Omega \times X : \lim_{n\to\infty}\frac{1}{n}\sum_{j=0}^{n-1}\delta_{F^j(\omega,x)}=\bar{\mu}_i \bigg\} = \bigcap_{\psi\in S} B(\bar{\mu}_i,\psi) \quad \text{for \(i=1,\dots,r\),}
\]
where the limit is taken in the weak* topology. Thus, item 3) in Definition~\ref{dfn:11} implies that the union of these statistical basins of attraction covers \(\Omega \times X\)  modulo \(\bar{m}\)-null set, i.e.,
\[
\bar{m}\big(B(\bar{\mu}_1)\cup\cdots\cup B(\bar{\mu}_r)\big)=1.
\]
\end{rem}

\begin{rem}[Physicality]
Since each \(B(\bar{\mu}_i)\) has full \(\bar{\mu}_i\)-measure and \(\bar{\mu}_i \ll \bar{m}\), it follows that \(\bar{m}(B(\bar{\mu}_i))>0\). Therefore, \(\bar{\mu}_i\) is a \emph{physical measure} with respect to the reference measure \(\bar{m}\) for every \(i=1,\ldots,r\).
\end{rem}

\begin{rem}[Disjointness] \label{rem:disjointness}
By Lemma~\ref{lem:abs}, the intersection of the measure-theoretic support of any pair of ergodic absolutely continuous invariant probability measures has zero reference measure. Consequently, the measures $\bar{\mu}_i$, $i=1,\dots,r$ in Definition~\ref{dfn:11} 
have pairwise disjoint supports modulo $\bar{m}$-null sets. Furthermore, since any ergodic $f$-stationary probability measure $\mu \ll m$ lifts to an 
ergodic $F$-invariant measure $\bar{\mu} = \mathbb{P} \times \mu$ with $\bar{\mu} \ll \bar{m}$, Lemma~\ref{lem:abs} also implies that Condition~2) in the definition of {\tt(FPM)} for a Bernoulli random map $f$ is automatically satisfied, making it redundant in the axiomatic formulation.
\end{rem}

In Theorem~\ref{prop:iid} below, we conclude that a Bernoilli random map  \(f\) satisfies {\tt (FPM)} with respect to \(m\) if and only if its associated skew-product \(F\) satisfies {\tt (FPM)} with respect to \(\bar{m}\). To get this, we require that the nonsingularity of the skew-product $F$. We recall that a transformation $g:Y\to Y$ of a probability space $(Y,\mathscr{G},\eta)$ is said to be \emph{$\eta$-nonsingular} if the preimage of any $\eta$-null set by $g$ is $\eta$-null. 

\begin{mainthm} \label{prop:iid}  
Let $(X , \mathscr B, m)$ and $(\Omega , \mathscr F, \mathbb{P})=(T^\mathbb{N}, \mathscr A^\mathbb{N}, p^\mathbb{N})$ be a Polish probability space and the infinite product space of a probability space $(T, \mathscr A, p)$, respectively.
Consider a measurable map $f:T\times X\to X$ and let $F:\Omega \times X \to \Omega \times X$ be its associated skew-product.  Assume that $F$ is $\bar{m}$-nonsingular, where $\bar{m}=\mathbb{P}\times m$.
Then, $f$ satisfies {\tt(FPM)} with respect to $m$ if and only if 
$F$ satisfies {\tt (FPM)} with respect to $\bar{m}$.
\end{mainthm}

The above theorem results from a more complete list of equivalences that generalizes~\cite[Theorem~C]{BNNT22} characterizing {\tt(FPM)} with several properties that describe the behavior of the Perron--Frobenius operators associated with both $f$ and $F$. See~Thorem~\ref{thm:BNNT}.    

~\vspace{-1cm}
\subsection{{\tt (FPM)} for Markovian random maps}
Here we study the property {\tt (FPM)} for the skew product $F$ associated with a Markovian random map, that is, for a random map $f:T\times X \to X$ driven by an ergodic shift-invariant Markov measure $\mathbb{P}$ on $(\Omega,\mathscr{F})=(T^\mathbb{N},\mathscr{A}^\mathbb{N})$. We denote by $Q(t,A)$ the transition probability of $\mathbb{P}$ and by $p$ its initial distribution.

In the Bernoulli case, for every $x_0 \in X$, the sequence $X_n(\omega)=f^n_\omega(x_0)$, $n\geq 1$ can be viewed as a  (time-homogeneous) discrete Markov chain  on $\Omega$ with  transition probability
\begin{equation} \label{eq:kernel-P}
P(x,A) =
 p(\{t\in T:  f_t(x)\in A\}) 
\quad \text{for } \ x\in X, \ A \in \mathscr{B}.
\end{equation}
For general Markovian random maps, the sequence $\{X_n\}_{n\geq 0}$ could no longer be a discrete (time-homogeneous) Markov chain. However, the sequence of random variables $\{\hat{X}_n\}_{n\geq 0}$ given by 
\begin{equation} \label{eq:Markov-chain}
\hat{X}_n(\omega) = (\omega_{n-1}, f^n_{\omega}(x_0)), \quad n\geq 1 \quad \text{and} \quad \hat{X}_0(\omega)=(t_0,x_0)
\end{equation}
is always a homogeneous Markov chain on $\Omega$ valued in the state space $T\times X$ and with transition probability given by
\begin{equation}\label{eq:Markov-transition}
\hat{P}((t,x), B)= \int 1_B( u,f_u(x))\, Q(t,du), \quad \text{for} \ \ (t,x)\in T\times X, \ \  B\in \mathscr{A}\otimes \mathscr{B}.
\end{equation}
In the Bernoulli case, the transition probability~\eqref{eq:Markov-transition} is a coupling of $p$ and $P(x,\cdot)$, that is, $\hat{P}((t,x), A\times X) =p(A)$ and $\hat{P}((t,x), T\times B)=P(x,B)$ 
where $P(x,\cdot)$ was introduced in~\eqref{eq:kernel-P}.

The following theorem extends \cite[Theorem~A and B]{BNNT22} to Markovian random maps. Before stating it, we need to recall some classical notions. 
A transition probability $K(x,A)$ is said to be \emph{Feller continuous} and \emph{strong Feller continuous}  if the family of probabilities
$K(x,\cdot)$ varies continuously with respect to the weak* topology and
setwise convergence on the space of probabilities, respectively. An operator $P$ is called \emph{quasi-compact} on a Banach space $E$ if there is a compact linear operator $R:E \to E$ such that $\|P^n-R\|_{\rm op}<1$ for some $n\in \mathbb{N}$. Hence, the class of quasi-compact operators contains the subclass of \emph{eventually compact operators}, that is, the linear operators $P$ such that $P^n$ is compact for some $n\in\mathbb{N}$. We also introduce a notion of nonsingularity for random maps. We say that  a random map $f:T\times X \to X$ is \emph{fibered $(p,m)$-nonsingular} if $f_t$ is $m$-nonsingular for 
$p$-a.e.~$t\in T$.

\begin{mainthm} \label{mainthmA} 
Let $(X, \mathscr{B}, m)$ and $(T, \mathscr{A} , p)$ be compact Polish probability spaces, where $p$ is the unique stationary measure of a transition probability $Q(t,A)$ on $T$.  Let $\mathbb{P}$ be the Markov measure with initial distribution $p$ and kernel $Q(t,A)$ and consider a fibered $(p,m)$-nonsingular random map $f:T\times X \to X$. Let $\hat{P}((t,x),B)$ be the transition probability given in~\eqref{eq:Markov-transition}. If  $\hat{P}^{n_0}((t,x),B)$ is strong Feller for some $n_0\in \mathbb{N}$, then 
\begin{enumerate}[leftmargin=0.5cm,label=$\bullet$, itemsep=0.1cm]
    \item $F$ satisfies {\tt (FPM)} with respect to $\mathbb{P}\times m$  and
    \item the operator $\hat P : L^\infty(p\times m) \to L^\infty(p\times m)$ given by
$$
   \hat{P}\varphi(t,x)=\int \varphi(\hat{u}) \, \hat{P}((t,x),d\hat{u}), \quad \varphi \in L^\infty(p\times m)
$$
is eventually compact (namely, $\hat P^{2n_0}$ is a compact operator).
\end{enumerate}
\end{mainthm}

The following result extends the physical noise condition in~\cite[Remark~1.4~(b)]{BNNT22} to the setting of Markovian random maps. See Definition~\ref{def-physical-noise}.

\begin{maincor} \label{maincorA}
  Let $(X,\mathscr{B},m)$, $(T,\mathscr{A},p)$, $Q(t,A)$ and $f:T\times X \to X$ be as in Theorem~\ref{mainthmA}.  If $\hat{P}^{n_0}((t,x),B)$ is Feller continuous and $\hat{P}^{n_0}((t,x),\cdot)\ll p\times m$ for all $(t,x)\in T\times X$ and some $n_0\in \mathbb{N}$,  then $\hat{P}^{n_0}((t,x),B)$ is strong Feller and the conclusions of Theorem~\ref{mainthmA} hold.  
\end{maincor}

The next corollary provides a practical example of an additive noise type that allows us to apply the above corollary. 
Notice that the following example appeared in~\cite[Example~1.1]{Kifer1986}.

\begin{maincor}  \label{maincor:D}
Let $X = T$ be a compact Lie group with its normalized Lebesgue measure~$m=p$ and~$Q(t,A)$ be a transition probability given by~$Q(t,A)=p (\{ s\in T : s+t\in A\})$.
For~$f(t,x)=f_0(x)+t$
with~$f_0:X\to X$  any $C^1$ diffeomorphism,
the conclusions of Theorem~\ref{mainthmA} hold with $n_0=2$.
\end{maincor}

We prove Corollary~\ref{maincor:D} by showing Theorem~\ref{thm:thm proof cor D}, a general version of the corollary.

\subsection{{\tt (fpm)} for mostly contracting Markovian random maps} \label{sec:lyapunov-def}

A natural question proposed in~\cite{BNNT22} was whether {\tt(FPM)} could be weakened by asking finitely many physical ergodic stationary measures instead of finitely many absolutely continuous ergodic stationary measures. Namely, it was introduced the property  \emph{{\tt (fpm)}  with respect to $m$ for a Bernoulli random map $f$} as the existence of finitely many ergodic $f$-stationary probability measures $\mu _1,\dots, \mu_r$ on $X$  such that
\begin{itemize} 
\item[1')] $\mu_i$ is a physical measure with
respect to $m$ for $i = 1,\dots,r$;
\item[2')]  they have pairwise disjoint supports (up to an $m$-null set);
\item[3')] $m(B_\omega(\mu_i)\cup \dots \cup B_\omega(\mu_r))=1$ for $\mathbb{P}$-a.e.~$\omega \in \Omega$, where 
\begin{equation*}\label{eq:basin-random}
\qquad B_\omega(\mu _i)=\bigg\{ x\in X : \, \lim _{n\to \infty} \frac{1}{n} 
\sum_{j=0}^{n-1} \delta_{f_{ \omega}^{j}(x)} = \mu _i
\bigg\} \quad \text{for $i=1,\dots,r$}.
\end{equation*}
\end{itemize}
In~\cite{Mal:17}, it was shown that a random map $f$ on a compact metric phase space $(X,d)$ having the property of \emph{local contraction} satisfies {\tt (fpm)}. This property means that for every $x \in X$, and for $\mathbb P$-a.e.~$\omega\in \Omega$, there is a neighborhood $B$ of $x$ such that $\mathrm{diam} f^n_\omega(B) \to 0$ as $n\to \infty$. Later in~\cite{BM24}, the authors showed that also the skew-product $F$ associated with $f$ satisfies the following property of finitude of physical measures in the Bernoulli case: 

\begin{defi}\label{dfn:22}
Let \(\mathbb{P}\) be an ergodic shift-invariant Markov measure on \(\Omega\). We say that the skew-product \(F\) associated with  \(f\) satisfies \emph{{\tt (fpm)} with respect to \(\bar{m}=\mathbb{P}\times m\)} if there is finitely many ergodic \(F\)-invariant probability measures \(\bar{\mu}_1,\dots,\bar{\mu}_r\) on \(\Omega \times X\) such that for each \(i=1,\ldots,r\)
\begin{enumerate}[label=\arabic*)]
  \item  \(\bar{\mu}_i\) is a \(\mathbb{P}\)-Markovian measure, i.e., \(d\bar{\mu}_i = \mu_{i,\omega_0}\,d\mathbb{P}(\omega)\);   
  \item $\bar{m}(B(\bar{\mu}_i))>0$, i.e., $\bar{\mu}_i$ is a physical measure with respect to $\bar{m}=\mathbb{P}\times m$; and
  \item $\bar{m}\left(B(\bar{\mu}_1)\cup \dots \cup B(\bar{\mu}_r)\right)=1$.
\end{enumerate} 
\end{defi}

From the remarks after Definition~\ref{dfn:11}, we have that {\tt (FPM)} implies {\tt (fpm)}.  An important class of Bernoulli random maps that satisfy the local contraction property is the \emph{mostly contracting random maps} introduced in~\cite{BM24}.  As a consequence, {\tt (fpm)} is satisfied for such random maps and their associated skew-product. To properly introduce this class of random maps, we need some definitions.  In what follows, $(X,d)$ denotes a compact metric space.

A function {$g:X\to X$} is said to be \emph{Lipschitz} if there is $C\geq 0$ such that $d(g(x),g(y))\leq C d(x,y)$. The infimum of $C\geq 0$ such that the above inequality holds is called \emph{Lipschitz constant} and denoted by~$\mathrm{Lip}(g)$.
The \emph{locally Lipschitz constant} of  $g$ at $z\in X$ is defined as
\begin{align*}
   Lg(z) \eqdef \inf_{r>0} L_rg(z) = \lim_{r\to 0^+} L_rg(z)
 \end{align*}
with
\begin{align*}
   L_rg(z)\eqdef \mathrm{Lip}(g|_{B(z,r)})=\sup\left\{   \frac{d(g(x),g(y))}{d(x,y)} \, : \ x,y\in B(z,r), \ x\not=y \right \},
\end{align*}
where $B(z,r)$ denotes the open ball centered at $z$ and of radius $r$.
We have that $\mathrm{Lip}(g)$ is greater than or equal to the supremum of $Lg(z)$ over all $z\in X$. 
Moreover, if $g$ is a diffeomorphism of a Riemannian manifold $X$, then $Lg(z) = \|Dg(z)\|$.

Let \(f:T\times X \to X\) be a random map driven by an ergodic shift-invariant Markov measure $\mathbb{P}$ on $\Omega=T^\mathbb{N}$.  We say that $f$ is \emph{Lipschitz random map} if $\mathrm{Lip}(f_{\omega_0})<\infty$ for $\mathbb{P}$-a.e.~$\omega\in \Omega$.  For each \(\omega \in \Omega\) and \(x \in X\), define \(L_n(\omega,x) = Lf_\omega^n(x)\). By the chain rule, we derive the inequality, $\log L_{n+m} \leq  \log L_n \circ F^m + \log L_m$. Hence, 
under the integrability condition
\begin{equation} \label{eq:integrability}
   \int \log^+ \mathrm{Lip}(f_{\omega_0} ) \, d\mathbb{P}(\omega) <\infty,
\end{equation}
for any $F$-invariant probability measure $\bar{\mu}$ on $\Omega\times X$ with first marginal $\mathbb{P}$,  by Kingman's subadditive ergodic theorem~\cite[Theorem~I.1]{ruelle1979ergodic}, there exists an $F$-invariant measurable function $\lambda: \Omega\times X \to \mathbb{R}\cup \{-\infty\}$ such that $\lambda^+\eqdef \max\{0,\lambda\}$ is $\bar{\mu}$-integrable,
\[
     \lambda(\omega,x)=\lim_{n\to\infty} \frac{1}{n} \log Lf^n_\omega(x) \quad \text{$\bar{\mu}$-almost every $(\omega,x)\in \Omega\times X$}.
\]
The \emph{maximal Lyapunov exponent} of \(\bar\mu\) is given by
$$
   \lambda(\bar{\mu})\eqdef \inf_{n \ge 1} \frac{1}{n} \int \log Lf^n_\omega(x)\, d\bar{\mu} = \lim_{n\to \infty} \frac{1}{n} \int \log Lf^n_\omega(x) \, \, d\bar{\mu} = \int \lambda(\omega,x) \, d\bar{\mu} \in [-\infty,\infty).
$$
If $X$ is a Riemannian manifold,
$\lambda(\bar{\mu})$ coincides with the maximal average Lyapunov exponent from Oseledets' theorem for the linear cocycle over the skew-product $F$.

\begin{defi} Let \(\mathbb{P}\) be an ergodic shift-invariant Markov measure on \(\Omega=T^\mathbb{N}\). 
A  random map $f:T\times X \to X$ is 
\emph{$\mathbb{P}$-mostly contracting} if the integrability condition~\eqref{eq:integrability} holds and 
$$
\lambda(f)\eqdef \sup \{\lambda(\bar{\mu}):  \ \bar{\mu} \ \text{is a $\mathbb{P}$-Markovian measure} \} <0.
$$
\end{defi}
Notice that the above notion of mostly contracting coincides with the definition previously introduced in~\cite[Definition~1.1]{BM24} for Bernoulli random maps.  
The following result extends one of the main results in~\cite[Theorem~B and item~(b) of Corollary~I]{BM24} for Markovian random maps. Recall that $C^{\alpha}(Y)$ denotes the Banach space of real-valued $\alpha$-H\"older continuous maps on a metric space $Y$, where $0<\alpha \leq 1$. 

\begin{mainthm} \label{cor:Markovian-qc} Let $\mathbb{P}$ be an ergodic shift-invariant Markov probability on $\Omega=T^\mathbb{N}$, where $T$ is a finite set and consider a compact metric space $(X,d)$. 
If  $f:T \times X \to  X$ is a $\mathbb{P}$-mostly contracting random map, then its associated skew-product  $F:\Omega \times X \to \Omega \times X$ satisfies {\tt (fpm)} with respect to $\bar{m}=\mathbb{P}\times m$, where $m$ is any reference probability measure on~$X$. 

Moreover, there is $\alpha_0 > 0$ such for  any $0 < \alpha \leq  \alpha_0$, the operator $\hat{P}:C^\alpha(T\times X)\to C^\alpha(T\times X)$ associated with the transition probability $\hat P((t,x),B)$ in~\eqref{eq:Markov-transition}, 
$$
   \hat{P}\varphi(t,x)=\int \varphi(\hat{u}) \, \hat{P}((t,x),d\hat{u}), \qquad \varphi\in C^\alpha(T\times X),
$$
is quasi-compact.
\end{mainthm}

In dimension one, as a consequence of the \emph{Markovian invariant principle} (see Theorem~\ref{thm:Markovian-inv-principle}), we can naturally apply the above theorem for non-degenerate Markovian random walks on the semigroup of $C^1$ diffeomorphisms of the circle $\mathbb{S}^1$ as shown in the following corollary.
Recall that, for a measurable map $g:X_1\to X_2$ and a measure $\nu$ on $X_1$, the pushforward measure of $\nu$ by $g$ is the measure on $X_2$ defined by $g\nu\left(E\right)=\nu(g^{-1}E)$ for a measurable set $E$ of~$X_2$.

\begin{maincor} \label{cor:mostly-contraction-circle}
Let $\mathbb{P}$ be an ergodic shift-invariant Markov probability on $\Omega=T^\mathbb{N}$ and consider a random map $f: T \times \mathbb{S}^1 \to \mathbb{S}^1$ of  $C^1$ diffeomorphisms of the circle such that $\log^+ \|f'_{\omega_0}\|$ is $\mathbb{P}$-integrable and there is no
measurable family of probability measures $\{\mu_t\}_{t\in T}$ on $\mathbb{S}^1$ such that 
\begin{equation} \label{eq:condition-no-comon-invariant}
    f_{\omega_0}\mu_{\omega_0}=\mu_{\omega_1} \quad \text{for $\mathbb{P}$-a.e.~$\omega=(\omega_i)_{i\geq 0} \in \Omega$}.
\end{equation}
Then $f$ is $\mathbb{P}$-mostly contracting. Moreover, if $T$ is finite, the conclusions in Theorem~\ref{cor:Markovian-qc} hold for $f$.
\end{maincor}

Since interval diffeomorphisms onto their image can be extended to circle diffeomorphisms, from the previous proposition, we obtain a similar result for the class of random maps of $C^1$ diffeomorphisms onto their image of a compact interval $\mathbb{I}$ of the real line. Moreover, in the case when $T$ is finite, we can also simplify~\eqref{eq:condition-no-comon-invariant} as follows. 
Recall that a transition probability $Q(t,A)$ is said to be \emph{irreducible} if there is a probability measure $\nu$ on $T$ such that for every set \( A \in \mathscr{A} \) with \( \nu(A) > 0 \), and for all \( t \in T \), there exists \( n \in \mathbb{N} \)  such that \( Q^n(t, A) > 0 \).

\begin{maincor} \label{cor:interval}
Let $T$ be a finite set and consider an ergodic shift-invariant Markov probability  $\mathbb{P}$ on $\Omega=T^\mathbb{N}$ with an irreducible transition probability $Q(t,A)$ on $T$. Let $f: T \times \mathbb{I} \to \mathbb{I}$ be a random map of $C^1$ interval diffeomorphisms onto its image. Assume that there is no collection of sets of two points $\{A_t\}_{t\in T} \subset \mathbb{I}$, such that
$f_{t}(A_t)=A_a$ if $Q(t,\{a\})>0$. 
Then $f$ is $\mathbb{P}$-mostly contracting and satisfies the conclusions of Theorem~\ref{cor:Markovian-qc}.
\end{maincor}

A random map $f$ in the assumption of the above corollary is just a one-dimensional confined nonlinear random walk in a finite tuple of diffeomorphisms onto its image of a compact interval driven by an irreducible Markov chain.
The condition on the collection of points can be read as the \emph{nonexistence of a finite invariant orbit} under the random dynamics. In the case that all the maps $f_t$ preserve orientation, the condition simplifies taking $x_t=y_t$, where $A_t=\{x_t,y_t\}$.  
The nonexistence of such a finite invariant orbit is a generic condition on the space of random maps of interval diffeomorphisms onto their image. Thus, for such one-dimensional confined generic non-linear Markovian random walks, Corollary~\ref{cor:interval} provides a 
geometric characterization of their ergodic stationary measures and proves that all of them have negative Lyapunov exponents. Moreover, these measures lift to physical measures for the associated canonical skew-products over a subshift of finite type and the union of its basin of attraction covers the full space modulo reference measure. A similar description of such systems was previously obtained by Kleptsyn and Volk in~\cite{kleptsyn2014physical} using totally different methods. In~\cite{kleptsyn2014physical}, the transition probability is assumed primitive (i.e., there exists $n\geq 1$ such $Q^n(t,\{a\})>0$ for every $t,a \in I$, or equivalently $Q(t,A)$ is irreducible and aperiodic) and some topological extra generic assumptions are included.
Although the authors only deal with maps $f_t$ that preserve orientation, they indicated in~\cite[Remark~1.4]{kleptsyn2014physical} that their result is applicable if one appropriately extends the skew-product $F$ to a new skew-product for which all fiber maps preserve orientation.
See~\cite{gelfert2020invariant}, where this strategy was implemented. The main novelty in Corollary~\ref{cor:interval}, in addition to simplifying the generic conditions and improving the requirement on the transition probability, is that we also obtain the quasi-compactness of the associated Koopman operator $\hat{P}:C^\alpha(T\times X)\to C^\alpha(T\times X)$ from which such properties (and many other statistical consequences) are obtained.

\subsection{Organization of the paper}  In Section~\ref{sec:FPM-thmA} we revisit the definition {\tt(FPM)} and prove a more complete version of Theorem~\ref{prop:iid} generalizing~\cite[Theorem~C]{BNNT22}.  Section~\ref{sec:iid-representation} establishes the theory that enables us to transfer results from Bernoulli random maps to Markovian random maps. 
In Section~\ref{sec:mean-constrictive idd respresentations} we prove Theorem~\ref{mainthmA}, Corollaries~\ref{maincorA} and~\ref{maincor:D}. On the other hand, Theorem~\ref{cor:Markovian-qc} is proved in Section~\ref{sec:Markov} and in Section~\ref{sec:invariant} we show in detail Corollaries~\ref{cor:mostly-contraction-circle} and~\ref{cor:interval}.
\section{Revisiting {\tt (FPM)}: Proof of Theorem~\ref{prop:iid}} \label{sec:FPM-thmA}
Let \((X, \mathscr{B}, m)\) be a Polish probability space and let \((\Omega, \mathscr{F}, \mathbb{P}) = (T^\mathbb{N}, \mathscr{A}^\mathbb{N}, p^\mathbb{N})\) be the infinite product of a probability space \((T, \mathscr{A}, p)\). We denote by \(\bar{m} = \mathbb{P} \times m\).

Let \(f: T \times X \to X\) be a random map. We say that  $f:T\times X \to X$ is \emph{fibered $(p,m)$-nonsingular} if $f_t$ is $m$-nonsingular for $p$-a.e.~$t\in T$. We also say that $f$ is \emph{$(p,m)$-nonsingular} if the preimage of any $m$-null set by  $f_t$ is $m$-null for $p$-a.e.~$t\in T$. That is, $f_tm(A) = 0$ for $p$-a.e.~$t\in T$ whenever $m(A) = 0$. Here, $t$ depends on $A$.  Clearly, if $f$ is fibered $(p,m)$-nonsingular, then $f$ is $(p,m)$-nonsingular. However, the converse is not true, in general.  Moreover, we have the following relation:

\begin{lem}\label{lem:nonsigular-relations}
Let \(F\) be the skew-product associated with the random map \(f\). Then
\[
 f \text{ is fibered } (p,m)\text{-nonsingular} \implies F \text{ is } \bar{m}\text{-nonsingular} \implies f \text{ is } (p,m)\text{-nonsingular}.
\]
\end{lem}
\begin{proof}
Assume that \(f\) is fibered \((p,m)\)-nonsingular. Hence, there exists a set \(T_0 \subset T\) with \(p(T_0)=1\) such that \(f_t m \ll m\) for every \(t \in T_0\). Consider a set \(A \subset \Omega \times X\) with \(\bar{m}(A)=0\). By Fubini's theorem, there exists a subset \(\Omega_0 \subset \Omega\) with \(\mathbb{P}(\Omega_0)=1\) such that
$m(A_\omega)=0$ for every $\omega \in \Omega_0$,
where
$A_\omega = \{ x \in X : (\omega,x) \in A \}$.
Let \(\Omega_1 = \Omega_0 \cap \sigma^{-1}(\Omega_0) \cap (T_0 \times \Omega)\). Since the Bernoulli measure \(\mathbb{P} = p^\mathbb{N}\) is invariant under the shift map \(\sigma\), we have \(\mathbb{P}(\Omega_1)=1\). Moreover, for every \(\omega \in \Omega_1\) we have \(f_{\omega_0} m(A_{\sigma(\omega)})=0\). Therefore,
\[
F\bar{m}(A) = \int f_{\omega_0} m(A_{\sigma(\omega)})\, d\mathbb{P}(\omega)
= \int_{\Omega_1} f_{\omega_0} m(A_{\sigma(\omega)})\, d\mathbb{P}(\omega)
= 0.
\]
This shows that \(F\) is \(\bar{m}\)-nonsingular.

Now assume that \(F\) is \(\bar{m}\)-nonsingular. Let \(A \subset X\) with \(m(A)=0\). Then, \(\bar{m}(\Omega \times A)=0\) and hence $F\bar{m}(\Omega \times A)=0$.
Moreover,
\[
\int f_t m(A) \, dp(t)= \int f_{\omega_0} m(A) \, d\mathbb{P}(\omega) = F\bar{m}(\Omega \times A)=0.
\]
This shows that \(f\) is \((p,m)\)-nonsingular.
\end{proof}

Associated with \(f\), we introduce the transition probability
\begin{equation*} 
P(x,A) = p(\{t\in T : f_t(x) \in A\})
= \mathbb{P}(\{\omega\in \Omega : f_\omega(x) \in A\}),
\quad \text{for } x\in X \text{ and } A \in \mathscr{B}.
\end{equation*}
We say that \(P(x,A)\) is \emph{\(m\)-nonsingular} if \(P(x,A)=0\) for \(m\)-a.e.~\(x \in X\) whenever \(m(A)=0\).

\begin{lem}\label{lem:f-nonsigular-P}
A random map \(f\) is \((p,m)\)-nonsingular if and only if its associated transition probability \(P(x,A)\) is \(m\)-nonsingular.
\end{lem}
\begin{proof}
Notice that
\begin{align*}
\int f_t m (A) \, dp(t)
&= \int 1_A(x) \, df_t m(x)\, dp(t)
= \int 1_A \circ f_t(x) \, dm(x) \, dp(t) \\
&= \int 1_A \circ f_t(x) \, dp(t) \, dm(x)
= \int P(x,A) \, dm(x).
\end{align*}
Thus, \(\int f_t m(A)\, dp(t)=0\) if and only if \(\int P(x,A)\, dm(x)=0\), establishing the desired equivalence.
\end{proof}

A $(p,m)$-nonsingular random map $f:T\times X \to X$ induces a linear bounded positive operator $\mathcal{L}_f: L^1(m) \to  L^1(m)$ given by
\begin{equation} \label{eq:Lf}
\mathcal{L}_f \varphi = \frac{d}{dm} \int f_tm_\varphi \, dp(t), \quad \text{where} \quad dm_\varphi = \varphi \, dm, \quad \varphi\in L^1(m)
\end{equation}
and $\frac{d}{dm}$ denotes the Radon--Nikod\'ym
derivative with respect to $m$. Here $L^1(m)=L^1(X,\mathscr{B},m)$ (resp.~$L^\infty(m)=L^\infty(X,\mathscr{B},m)$) denotes the Banach space of measurable functions $\varphi:X\to\mathbb{R}$ with \(\|\varphi\| =\int_X|\varphi|\,dm<\infty\) (resp.~$\|\varphi\|_\infty= \mathrm{ess}\sup |\varphi| < \infty$), where functions that agree \(m\)-almost everywhere are identified.  We call $\mathcal{L}_f$ the \emph{annealed Perron--Frobenius operator}. 
Moreover, if $f$ is fibered $(p,m)$-nonsingular, then 
$$
\mathcal{L}_f \varphi = \int \mathcal{L}_t \varphi \, dp(t) \quad \text{for }\varphi\in L^1(m) ,
$$
where $\mathcal{L}_t\varphi =\frac{d}{dm} f_tm_\varphi$ denotes the Perron--Frobenius associated with $f_t$.  
See below and \cite[Appendix~A]{BNNT22} for more details.

\begin{rem}\label{rem:def-Perron}
Following \cite{ito1964invariant}, an \(m\)-nonsingular transition probability induces a bounded linear positive operator \(\mathcal{L}_f : L^1(m) \to L^1(m)\) defined via the Radon--Nikod\'ym derivative by
\begin{equation} \label{eq:Lf2}
\mathcal{L}_f\varphi = \frac{d}{dm} \left( \int P(x,\cdot)\, dm_\varphi \right), \quad \text{where } dm_\varphi = \varphi\, dm\text{ and } \varphi \in L^1(m).
\end{equation}
Since \(P(x,\cdot) = f^xp\), where \(f^x = f(\cdot,x)\), we have
\[
\int P(x,A)\, dm_\varphi = \int 1_{(f^x)^{-1}(A)} \, dp \, dm_\varphi = \int 1_A \, d\bigl(f_t m_\varphi\bigr) \, dp = \int f_t m_\varphi(A) \, dp.
\]
Thus, the expression in \eqref{eq:Lf2} coincides with that in \eqref{eq:Lf}. Moreover, \(\mathcal{L}_f\) is a \emph{Markov operator}, i.e., a positive linear operator satisfying \(\mathcal{L}^*_f1_X = 1_X\), where the dual operator \(\mathcal{L}^*_f: L^\infty(m) \to L^\infty(m)\) is given by
\[
\mathcal{L}^*_f\varphi = \int \varphi(y)\, P(x,dy), \quad \varphi \in L^\infty(m).
\]
We also have that the iterations $\mathcal{L}^{*n}_f$  are related with the $n$-th transition probability 
$P^n(x,A)$ in the following way
\begin{equation} \label{eq:L*}
    \mathcal{L}^{*n}_f\varphi = \int \varphi(y) \, P^n(x,dy)=\int \varphi (f^n_\omega(x))\, d\mathbb{P}(\omega) \quad \varphi \in L^\infty(m).
\end{equation}
\end{rem}

Under the assumption that $f$ is fibered $(p,m)$-nonsingular, the property {\tt (FPM)} is characterized by the mean constrictivity of the annealed Perron--Frobenius operator $\mathcal L_{f}$~\cite[Theorem~C]{BNNT22}.  We recall that the operator $ {\mathcal{L}}_f$ is said to be \emph{mean constrictive} (or {\tt (MC)} for short), if there is a compact set $F$ of $L^1(m)$ such that $d(A_n h,F) \to 0$ as $n\to \infty$ for any $h \in D(m)$,  where  
 \[
 d(\varphi,F)=\inf_{\psi\in F}\|\varphi -\psi \| \qquad \text{and} \qquad 
A_n\varphi = \frac{1}{n}\sum_{i=0}^{n-1}\mathcal{L}_f^i\varphi \qquad \text{for $\varphi \in L^1(m)$}
\]
and  $D(m)=D(X,\mathscr{B},m)$ is the space of density functions. That is, $h\in D(m)$ if and only if $h \geq 0$ $m$-almost everywhere and $\Vert h \Vert =1$.

We define the support of a real-valued function \(h\) (up to an \(m\)-null set) by
\[
\supp h = \{ x \in X : h(x) \neq 0 \}.
\]
We say that \(h\) is an \emph{invariant density} of \(\mathcal{L}_f\) (or a \(\mathcal{L}_f\)-invariant density) if \(h \in D(m)\) and \(\mathcal{L}_f h = h\). We also say that \(h\) is \emph{ergodic} if the probability measure \(m_h\) defined by \(dm_h = h\,dm\) is ergodic, that is, if \(m_h(A) \in \{0,1\}\) for any \(A \in \mathscr{B}\) satisfying \(\mathcal{L}_f^*1_A \geq 1_A\).

Following \cite{BNNT22},  \(\mathcal{L}_f\) satisfies the property of \emph{finitude of ergodic densities} (or {\tt (FED)} for short) if \(\mathcal{L}_f\) admits finitely many ergodic \(\mathcal{L}_f\)-invariant densities \(h_1,\dots,h_r\) with mutually disjoint supports (up to an \(m\)-null set) and if the invariant density
\[
h = \frac{1}{r}(h_1 + \dots + h_r)
\]
has the maximal support; that is,
\begin{align*}
\lim_{n \to \infty} \mathcal{L}_f^{n*}1_{\supp h}(x) = 1 \quad \text{for \(m\)-a.e.~\(x \in X\)}.
\end{align*}
Here, \(\mathcal{L}^{n*}_f\) denotes the dual operator of $\mathcal{L}^n_f=\mathcal{L}_f \circ \overset{(n)}{\dots} \circ \mathcal{L}_f$.

On the other hand, \(\mathcal{L}_f\) is said to be \emph{almost periodic in mean} (or {\tt (APM)} for short) if there exist finitely many ergodic \(\mathcal{L}_f\)-invariant densities \(h_1,\dots,h_r\) with mutually disjoint supports (up to an \(m\)-null set) and positive bounded linear functionals \(\lambda_1, \dots, \lambda_r\) on \(L^1(m)\) such that
\[
\lim_{n\to\infty} \left\| A_n\varphi-\sum_{i=1}^r\lambda_i(\varphi)h_i \right\| = 0 \quad \text{for all } \varphi\in L^1(m).
\]

 If \(h\) is an ergodic \(\mathcal{L}_f\)-invariant density, then \(m_h\) is an ergodic \(f\)-stationary measure. Consequently, the measure \(\bar{m}_h = \mathbb{P} \times m_h\) is \(F\)-invariant and absolutely continuous with respect to \(\bar{m} = \mathbb{P} \times m\). 
 In view of this observation and as a consequence of the following lemma, we have:

\begin{rem}
The requirement that the supports of the densities \(h_i\) (for \(i=1,\dots,r\)) be disjoint (up to an \(m\)-null set) is redundant in the definitions of {\tt (FED)} and {\tt (APM)}.
\end{rem}

Although the following result can be derived from~\cite[Proposition~C.10 and Remark~C.11]{BNNT22} or~\cite[Theorem~1.4.8]{aaronson1997introduction}, we provide a different direct proof here.

\begin{lem}\label{lem:abs}
Let \((Y, \mathscr{G}, \eta)\) be a probability space and consider a measurable map \(g: Y \to Y\). Suppose that \(\mu_1\) and \(\mu_2\) are two distinct ergodic \(g\)-invariant probability measures that are both absolutely continuous with respect to \(\eta\). Then supports of \(\mu_1\) and \(\mu_2\)  are disjoint modulo an \(\eta\)-null set.
\end{lem}

\begin{proof} 
Since \(\mu_i\ll \eta\), there is \(h_i\in D(\eta)\) such that $d\mu_i=h_i\,d\eta$. Assume that $B=\supp h_1 \cap \supp h_2$
has positive \(\eta\)-measure. Consider the probability measure \(\eta_B\) 
defined by
\[
\eta_B(A)=\frac{\eta(A\cap B)}{\eta(B)} \quad \text{for } A\in \mathscr{G}.
\]
Since \(\mu_1\neq \mu_2\), there exists a set \(A\in \mathscr{G}\) such that \(\mu_1(A)\neq \mu_2(A)\). For each \(i=1,2\), let
$G_i = B(\mu_i,\varphi)$
be the set of \(\mu_i\)-generic points of \(g\) for \(\varphi=1_A\); that is, \(G_i\) is the set of points \(x\in Y\) for which the Birkhoff average
\[
\frac{1}{n}\sum_{j=0}^{n-1} \varphi\bigl(g^j(x)\bigr)
\]
converges to \(\int \varphi \,d\mu_i\) as \(n\to \infty\). Since 
$\int \varphi \,d\mu_1 \neq \int \varphi\, d\mu_2$, 
we have that \(G_1\cap G_2=\emptyset\). Moreover, by the Birkhoff ergodic theorem and the ergodicity of \(\mu_i\), we have \(\mu_i(G_i)=1\). Hence, defining 
$\tilde{G}_i := G_i\cap \supp h_i$,
we obtain \(\mu_i(\tilde{G}_i)=1\). Observing that 
$
\int_{\tilde{G}_i} h_i\, d\eta = \mu_i(\tilde{G}_i)=1,
$
we have 
\begin{equation}\label{eq:tildeG}
\eta(\tilde{G}_i)=\eta(\supp h_i) \quad \text{for $i=1,2$}.   
\end{equation}
This implies that $\eta(\tilde G_i \cap B)=\eta(B)$  for \(i=1,2\). Indeed, otherwise $\eta(\tilde G_i \cap B)<\eta(B)$ and since \(\tilde{G}_i\subset \supp h_i\),  it follows that 
\begin{align*}
\eta(\tilde G_i) &= \eta( \tilde G_i \cap \supp h_i) = \eta(\tilde G_i \cap B) + \eta(\tilde G_i \cap (\supp h_i \setminus B))  \\ &<   
\eta(B) +\eta(\supp h_i \setminus B)=\eta(\supp h_i)
\end{align*}
contradicting~\eqref{eq:tildeG}.  Thus, $\eta_B(\tilde G_i)=1$ for $i=1,2$ and consequently, 
$\eta_B(\tilde{G}_1\cap \tilde{G}_2)=1$,
which implies \(G_1\cap G_2\neq \emptyset\), a contradiction.
\end{proof}

We now present a more complete version of Theorem~\ref{prop:iid}. Moreover, in view of Lemma~\ref{lem:nonsigular-relations}, the following result generalizes also~\cite[Theorem~C]{BNNT22}.

\begin{thm}\label{thm:BNNT}
Let \(f:T\times X \to X\) be a random map driven by a Bernoulli measure \(\mathbb{P}\). Assume that its associated skew-product \(F\) is \(\bar{m}\)-nonsingular, where \(\bar{m}=\mathbb{P}\times m\). The following are equivalent:
\begin{multicols}{2}
\begin{enumerate}[label=(\roman*), leftmargin=1cm]
\item \({\mathcal{L}}_f\) is {\tt (MC)}; 
\item \(\mathcal{L}_f\) satisfies {\tt (FED)}; 
\item \(\mathcal{L}_f\) is {\tt (APM)};
\item \(f\) satisfies {\tt (FPM)} with respect to \(m\);
\item \({\mathcal{L}}_F\) is {\tt (MC)}; 
\item \(\mathcal{L}_F\) satisfies {\tt (FED)}; 
\item \(\mathcal{L}_F\) is {\tt (APM)};
\item \(F\) satisfies {\tt (FPM)} with respect to \(\bar{m}\).
\end{enumerate}
\end{multicols}
\end{thm}

The following proposition shows (ii)~\( \implies \)~(vi) in the above theorem.

\begin{prop}\label{APM-FPM-skew}
If \(F\) is \(\bar{m}\)-nonsingular and \(\mathcal{L}_f\) is {\tt (FED)}, then \(\mathcal{L}_F\) is {\tt (FED)}.
\end{prop}

\begin{proof} 
Let $h_1,\dots,h_r$ be ergodic invariant densities of $\mathcal{L}_f$ as in {\tt (FED)}. Hence, the measure $\mu_i$ given by $d\mu_i= h_i \, dm$, is an ergodic $f$-stationary measure. Equivalently, $\bar\mu_i=\mathbb{P}\times \mu_i$ is an ergodic $F$-invariant measure. This implies that the density  $1_{\Omega}\otimes h_i$ is an ergodic $\mathcal{L}_F$-invariant density for $i=1,\dots,r$.
We can further see $1_{\Omega}\otimes h$ has the maximal support, where $h=\frac{1}{r}(h_1+\dots+h_r)$.
Indeed,
by~\eqref{eq:L*}, for any $A\in\mathscr{B}$ we have
\[
\int \mathcal{L}_F^{*n}1_{\Omega\times A}\, d\mathbb{P} =
\int 1_A\circ  f^n_{\omega}\, d\mathbb{P}(\omega) =
\mathcal{L}_f^{*n}1_A
\]
$m$-almost everywhere.
Since $h$ has the maximal support, we have
\begin{align*}
\int\lim_{n\to\infty}\mathcal{L}_F^{*n}1_{\Omega\times (X\setminus\supp h)}\,d\mathbb{P}
&=\lim_{n\to\infty}\int\mathcal{L}_F^{*n}1_{\Omega\times (X\setminus\supp h)}\,d\mathbb{P}
=\lim_{n\to\infty}\mathcal{L}_f^{*n}1_{X\setminus\supp h}
=0
\end{align*}
$m$-almost everywhere.
Here we used the fact that $\mathcal{L}_F^{*n}1_{\Omega\times (X\setminus\supp h)}$ and $\mathcal{L}_f^{*n}1_{X\setminus \supp h}$ are decreasing and the Lebesgue dominated convergence theorem.
Thus $\lim_{n\to\infty}\mathcal{L}_F^{*n}1_{\Omega\times (X\setminus\supp h)}=0$ $\bar m$-almost everywhere and $1_{\Omega}\otimes h$ has the maximal support.
Therefore, we conclude that ergodic $\mathcal{L}_F$-invariant densities $1_{\Omega}\otimes h_1,\dots,1_{\Omega}\otimes h_r$ satisfy the condition ({\tt FED}) for $\mathcal{L}_F$.
\end{proof}

\begin{proof}[Proof of Theorem~\ref{thm:BNNT} (Theorem~\ref{prop:iid})]
The equivalence between (i)--(iii) follows from \cite[Theorem~D]{BNNT22}. In fact, this equivalence only requires that \(f\) be \((p,m)\)-nonsingular in order for the annealed Perron--Frobenius operator \(\mathcal{L}_f\) to be well defined. The implication (iv) \(\implies\) (ii) also follows under this weaker assumption (see \cite[Proposition~5.7]{BNNT22}). 

For the deterministic system $F$ we also get from~\cite[Theorem~C (and Theorem~D)]{BNNT22} that (v)--(viii) are equivalent under the assumption that $F$ is $\bar{m}$-nonsigular. We also have under this assumption that (ii)\(\implies\) (vi) from Proposition~\ref{APM-FPM-skew}.

Finally, to conclude the theorem, we show that (viii) \(\implies\) (iv). To show this,  since  $\mathbb{P}$ is a Bernoulli measure,  we have from Lemma~\ref{lem:Markov-Bernoulli} that $\mathbb{P}$-Markovian measures are product measures with first marginal $\mathbb{P}$. Hence, in this case,  the $F$-invariant measure $\bar\mu_i$ in condition 1) of Definition~\ref{dfn:22} are of the form $\bar\mu_i=\mathbb{P}\times \mu_i$. Moreover, since they are $F$-invariant probability measures absolutely continuous with respect to $\bar{m}=\mathbb{P}\times m$, it follows that $\mu_i$ is $f$-stationary absolutely continuous with respect to $m$. This concludes that 1) and 2) in {\tt(FPM)} for $F$ implies 1) in {\tt (FPM)} for $f$. Furthermore, as mentioned in Remark~\ref{rem:disjointness}, condition~2) in {\tt(FPM)} for $f$ is automatically satisfied by the absolute continuity of the stationary measures.  Also, since $B_\omega(\mu_i,\psi)$ coincides with the $\omega$-section of $B(\bar{\mu}_i,\varphi)$, where $\varphi(\omega,x)=\psi(x)$, we get by Fubini's theorem that condition~3)  in Definition~\ref{dfn:11} implies 3) in {\tt(FPM)} for $f$.  This completes the proof of the theorem.
\end{proof}

\section{iid~representation of Markovian random maps} \label{sec:iid-representation}
In this section, we aim to establish a comprehensive theory that enables the transfer of results from Bernoulli random maps to Markovian random maps. We start by fixing and recalling some notation.

Let \( (T, \mathscr{A})\) and \((X, \mathscr{B})\) be standard Borel spaces and denote $\hat X=T\times X$ and $\Omega =T^\mathbb{N}$.  
Let~$Q(t,A)$ be a transition probability on $T\times \mathscr{A}$ with a 
$Q$-stationary probability measure~$p$ on~$T$ and denote by~$\mathbb{P}$ its associated Markov probability measure on $\Omega$. We also denote by~$q(t,A)$ the transition probability defined as the \emph{dual} of~$Q(t,A)$. That is, the essentially unique transition probability such that 
$$
\int_B Q(t,A) \, dp(t) = \int_A q(t,B) \, dp(t) \quad \text{for any $A,B\in \mathscr{A}$.}
$$
See~\cite[Lemma~4.7 and Theorem~4.9]{Revuz84}  for more details.

According to the random mapping representation theorem (cf.~\cite[Proposition~1.5]{levin2017markov} and \cite[Theorem~1.1]{kifer-ergodic}), there exists a random map \( g: S \times T \to T \), where \( S = [0,1] \) is equipped with the normalized Lebesgue measure \( \rho \), such that
\begin{equation}\label{eq:representation}
  Q(t,A) = \rho(\{s \in S : g_s(t) \in A\}).
\end{equation}
We denote by \( \nu = \rho^{\mathbb{N}} \) the Bernoulli measure on the sample space \( \Sigma = S^{\mathbb{N}} \).

Let $f:T\times X \to X$ be a random map. We consider the measurable map 
\begin{equation} \label{eq:idd-representation}
 h: S\times \hat X \to \hat  X, \quad h_s(t,x)=(g_s(t),f_{g_s(t)}(x)).   
\end{equation}
Using the representation~\eqref{eq:representation}, the transition probability $\hat{P}((t,x),B)$ given in~\eqref{eq:Markov-transition} satisfies
\begin{align} \label{eq:P-h}
\hat{P}((t,x),B)=\int 1_B(u,f_{u}(x))\, Q(t,du)=\int 1_B(g_s(t),f_{g_s(t)}(x))\, d\rho(s) =\rho(\{s\in S:  h_s(t,x)\in B\}).
\end{align} 
Thus, the Markov chain generated by the random iterations  $\{h^n_\xi(t,x)\}_{n\geq 0}$, $\xi\in \Sigma$  has transition probability $\hat{P}((t,x),B)$.   
The random map $h$ is referred to as the \emph{iid representation} for the Markovian iterations of $f$. 

We denote by 
$$
H:\Sigma\times \hat X \to  \Sigma\times \hat X, \quad H(\xi,(t,x))=(\sigma(\xi), h_{\xi_0}(t,x)) 
$$
the skew-product associated with $h$.  The following result from~\cite[Lemma~3.1]{matias:2022} related $H$ with the skew-product $F$  associated with the random map $f$. 

\begin{lem} \label{lem:matias} There is a measurable map $\pi:\Sigma\times \hat X \to \Omega\times X$ given by
$$
  \pi(\xi,(t,x))=(\omega,x), \quad \text{where} \ \  \xi \in \Sigma, \ (t,x)\in \hat X \ \ \text{and} \ \ \omega=(g^k_\xi(t))_{k\geq 1} \in \Omega
$$
such that $\pi\circ H =  F \circ \pi$. Moreover, for every probability measure $\hat{\mu}$ on $T\times X$ with first marginal $p$, i.e., $d\hat{\mu}=\mu_t\, dp(t)$, it holds that $\bar{\mu}:= \pi(\nu\times \hat{\mu})$ is a probability measure on $\Omega\times X$ with first marginal $\mathbb{P}$ and  disintegration of the form 
$$d\bar{\mu}=\bar{\mu}_{\omega_0} \, d\mathbb{P}(\omega), \ \ \text{where} \ \ \bar{\mu}_{\omega_0}=\int \hat\mu_t \, q(\omega_0,dt).$$ 
In particular, $\pi(\nu\times(p\times m) )=\mathbb{P}\times m$ for any probability measure $m$ on $X$.  
\end{lem}

We will use this conjugacy between the skew-products associated with the Markov random map and its iid representation to transfer some results obtained for random maps driven by Bernoulli measures to the more general case of Markov noise.

From now on, the following standing assumption is needed to guarantee the correspondence between the stationary measures below.

\paragraph{\bf Standing Assumption:}
The transition probability \( Q(t, A) \) is \emph{uniquely ergodic}. This means that there is exactly one $Q$-stationary probability measure \( p \) on $T$.

Matias in~\cite[Theorem~2.1]{matias:2022} also characterizes the stationary measures and establishes the following relation between these measures and the Markovian probabilities: 

\begin{thm} \label{thm:Matias} Let $\hat \mu$ be a probability measure on $\hat X$. The following are equivalent
\begin{enumerate}[itemsep=0.1cm]
    \item $\hat{\mu}$ is $h$-stationary (i.e.,  $\hat\mu  = \int h_s \hat\mu \, d\rho(s)$); 
    \item $\nu\times \hat{\mu}$ is $H$-invariant (i.e., $H(\nu\times \hat{\mu})=\nu\times \hat{\mu}$);
    \item $\hat\mu$ is $\hat{P}$-stationary (i.e., $\hat\mu  = \int \hat P((t,x),\cdot)\, d\hat\mu(t,x)$);
    \item $d\hat \mu =\hat\mu_t\, dp(t)$, where $\hat{\mu}_t=\int f_t \hat{\mu}_s \, q(t,ds)$ for $p$-a.e.~$t\in T$.  
\end{enumerate}
 Moreover, the functions 
 $$
    d\bar{\mu}= \bar{\mu}_{\omega_0} \, d\mathbb{P}(\omega) \xmapsto{ \ \Theta \ }  d\hat{\mu}= (f_t\bar{\mu}_t )  \, dp(t) \quad \text{and} \quad d\hat{\mu}= \hat{\mu}_{t} \, dp(t) \xmapsto{\ \Xi \ } 
  d\bar{\mu} = \big(\int \hat \mu_t \, q(\omega_0,dt)\big)\, d\mathbb{P}(\omega)
 $$
 map, respectively, an (ergodic) $F$-invariant $\mathbb{P}$-Markovian probability measure $\bar\mu$ on $\Omega \times X$  to an (ergodic) $\hat P$-stationary  probability measure $\hat \mu$ on $\hat X$ and vice-versa. 
 \end{thm}

\begin{rem} \label{eq:mu-hat-bar} In view of Lemma~\ref{lem:matias} and Theorem~\ref{thm:Matias}, the function $\Xi$ can be extended for any probability measure $\hat{\mu}$ on $\hat X$ with first marginal $p$ as follows:
\begin{equation*} 
\Xi(\hat\mu) = \pi(\nu\times \hat\mu).
\end{equation*}
Moreover, $\Xi(\hat\mu)$ is still a $\mathbb{P}$-Markovian measure although it may be not $F$-invariant. The $F$-invariance follows when $\hat \mu$ is $h$-stationary (or equivalently $\hat P$-stationary or $\nu\times \hat \mu$ is $H$-invariant). 
\end{rem}

In the study of Markovian random iterations, alternatively to~\eqref{eq:Markov-chain}, we can consider
the Markov chain $\{\tilde{X}_n\}_{n\geq 0}$, where
$\tilde X_n(\omega)= (\omega_n, f^n_\omega(x))$ for $n\geq 0$.
In this case, the transition probability is given by
$$
\tilde{P}((t, x), B)= \int 1_B(s,f_t(x))\, Q(t,ds), \quad  \text{for} \ (t,x) \in T\times X, \ \ B\in \mathscr{A}\otimes \mathscr{B}.
$$
Similarly, Matias also stated the following result for $\tilde{P}$-stationary measures (see also~\cite{crauel1991markov}).

\begin{thm} \label{thm:Matias2} Let $\tilde \mu$ be a probability measure on $\hat X$. The following are equivalent
\begin{enumerate}[itemsep=0.1cm]
    \item $\tilde\mu$ is $\tilde{P}$-stationary (i.e., $\tilde\mu  = \int \tilde P((t,x),\cdot)\, d\tilde\mu(t,x)$);
    \item $d\tilde \mu =\tilde\mu_t\, dp(t)$, where $\tilde{\mu}_t=\int f_s \tilde{\mu}_s \, q(t,ds)$ for $p$-a.e.~$t\in T$.  
\end{enumerate}
 Moreover, there is a one-to-one correspondence between the set of (ergodic) $\tilde{P}$-stationary measures on $\hat X$ and the set of (ergodic) $F$-invariant $\mathbb{P}$-Markovian measures on $\Omega\times X$. Namely,
 $$
   d\bar\mu=\bar \mu_{\omega_0} \, d\mathbb{P}(\omega) \mapsto d\tilde{\mu}=\bar\mu_t \, dp(t)
 \qquad \text{and} \qquad
   d\tilde\mu=\tilde \mu_{t} \, dp(t) \mapsto d\bar{\mu}=\tilde\mu_{\omega_0} \, d\mathbb{P}(\omega).
 $$
\end{thm}

The one-to-one relation between the set $\mathcal{S}(\tilde P)$ of $\tilde{P}$-stationary measures and the set $\mathcal{M}$ of $F$-invariant $\mathbb{P}$-Markovian measures is evident from the definition of $\tilde \mu$ and $\bar\mu$ in the above theorem. Matias also claimed, without proof, that there is a one-to-one correspondence between the set $\mathcal{S}(\hat P)$ of $\hat{P}$-stationary measures and $\mathcal{M}$. However, the injectivity of both maps $\Theta$ and $\Xi$ in Theorem~\ref{thm:Matias} is not immediate. In the following proposition, we discuss this issue. 

\begin{prop} \label{prop:one-to-one}
The function $\Xi : \mathcal{S}(\hat P) \to \mathcal{M}$ is bijective and  
$\Theta \circ \Xi = \mathrm{id}_{\mathcal{S}(\hat P)}$. 
In particular,  $\Theta  = \Xi^{-1}$. 
\end{prop}
\begin{proof}
    Let $\hat\mu \in \mathcal{S}(\hat P)$. 
    By the equivalences shown in Theorem~\ref{thm:Matias}, we have that 
    \begin{equation} \label{eq:pp1}
    d\hat{\mu}= \hat \mu_t \, dp(t), \quad \text{where} \quad \hat\mu_t = \int f_t\hat{\mu}_s, q (t, ds).   
    \end{equation}
    Consider $\bar\mu = \Xi(\hat\mu)$ which,  according again to Theorem~\ref{thm:Matias}, belongs to $\mathcal{M}$ and satisfies 
    \begin{equation} \label{eq:pp2}
         d\bar{\mu} = \bar\mu_{\omega_0} \, d\mathbb{P}(\omega), \quad \text{where} \quad \bar\mu_{\omega_0}=\int \hat \mu_s \, q(\omega_0, ds).
    \end{equation}
Hence, by~\eqref{eq:pp2} and~\eqref{eq:pp1}, we have that 
\begin{equation} \label{eq:pp3}
f_t\bar{\mu}_t = \int f_t \hat\mu_s \, q(t,ds) = \hat\mu_t.
\end{equation}
Denote by $\hat \mu ^* = \Theta(\bar\mu)=(\Theta\circ \Xi)(\hat\mu)$. Then, by the definition of $\Xi$ and~\eqref{eq:pp3},  
    $
    d\hat\mu^* = (f_t\bar{\mu}_t) \, dp(t) = \hat\mu_t \, dp(t) = d\hat\mu. 
    $ 
 This implies that $\hat\mu^*=\hat\mu$ and concludes that $\Theta \circ \Xi = \mathrm{id}_{\mathcal{S}(\hat P)}$. From here, we easily get that $\Xi:\mathcal{S}(\hat P) \to \mathcal{M}$ is injective. Indeed, if $\Xi(\hat\mu)=\Xi(\hat\mu^*)$ for some $\hat\mu$ and $\hat\mu^*$ in $\mathcal{S}$, then $\hat\mu = \Theta \circ \Xi (\hat\mu)=\Theta \circ \Xi (\hat\mu^*)=\hat\mu^*$.

Now, we prove that $\Xi\circ \Theta =\mathrm{id}_{\mathcal{M}}$. To do this, according to Theorem~\ref{thm:Matias}, if $d\bar\mu=\bar\mu_{\omega_0}\,d\mathbb{P}(\omega)$, then $\hat\mu=\Theta(\mu)$ satisfies $d\hat\mu= (f_t\bar\mu_t)\, dp(t)$. Consider now $\bar\mu^*=\Xi(\hat\mu)$. Hence,  
\begin{equation} \label{eq:tt1}
    d\bar\mu^*= \big(\int f_t\bar\mu_t \, q(\omega_0,dt)\big)\, d\mathbb{P}(\omega).
\end{equation} 
On the other hand, consider also the $\tilde{P}$-stationary measure $\tilde \mu$ associated with $\bar\mu$ given by Theorem~\ref{thm:Matias2}. Namely, 
$d\tilde \mu = \tilde{\mu}_t  \, dp(t)$, where $\tilde{\mu}_t=\bar{\mu}_t$. Since this measure is $\tilde{P}$-stationary, by Theorem~\ref{thm:Matias2} we also have that 
\begin{equation} \label{eq:tt2}
    \bar \mu _t = \tilde \mu_t = \int f_s \tilde{\mu}_s\, q(t,ds) = \int f_s \bar{\mu}_s\, q(t,ds).
\end{equation}
Comparing~\eqref{eq:tt1} with~\eqref{eq:tt2}, we arrive that $d\bar\mu^*=\bar \mu_{\omega_0}\,d\mathbb{P}=d\bar\mu$.  Thus, $\mu^*=\bar\mu$ concluding that $\Xi\circ \Theta =\mathrm{id}_{\mathcal{M}}$. From here we get hat $\Xi$ is surjective since for every $\bar\mu$, $\Theta(\bar\mu)\in \mathcal{S}$ and $\Xi(\Theta(\bar\mu))=\bar\mu$. This concludes the proof of the proposition. 
\end{proof}
\begin{rem} As consequence of Theorems~\ref{thm:Matias} and~\ref{thm:Matias2} and the one-to-one relation shown in Proposition~\ref{prop:one-to-one}, we have the following one-to-one correspondence between stationary measures of $\hat{P}((t,x),B)$ and $\tilde P((t,x),B)$:
$$
\Phi: \mathcal{S}(\hat P) \to \mathcal{S}(\tilde P),  \quad d\hat\mu = \hat \mu_t \, dp(t) \mapsto d\tilde \mu = \big(\int \hat\mu_s \, q(t,ds) \big)\, dp(t)
$$
and
$$
\Phi^{-1}: \mathcal{S}(\tilde P) \to \mathcal{S}(\hat P), \quad d\tilde\mu = \tilde \mu_t \, dp(t) \mapsto d\hat \mu =  f_t\tilde\mu_t \, dp(t).
$$
\end{rem}

As a consequence of these results, we get the following:

\begin{lem} \label{lem:Markov-Bernoulli} Let $\bar \mu$ be an $F$-invariant $\mathbb{P}$-Markovian probabilty measure on $\Omega \times X$, where $\mathbb{P}$ is a  Bernoulli measure. Then, 
 there exists a probability measure $\mu$ on $X$ such that $\bar\mu=\mathbb{P}\times \mu$.    
\end{lem}
\begin{proof} Let $d\bar\mu = \bar \mu_{\omega} \, d\mathbb{P}(\omega)$. By Theorem~\ref{thm:Matias2}, the measure $\tilde{\mu}$ on $\hat X$ given by $d\tilde\mu= \bar{\mu}_t\, dp(t)$ is a $\tilde{P}$-stationary measure, where $p$ is the initial distribution on $T$ of $\mathbb{P}$. Then, according to the characterization of such stationary measures,
$$
\bar{\mu}_t = \int f_s \bar\mu_s \, q(t,ds) \quad \text{for $p$-a.e.~$t\in T$.}
$$
However, since $\mathbb{P}$ is a Bernoulli measure,  $\mathbb{P}=p^\mathbb{N}$, $Q(t,\cdot)=p$ and $q(t,\cdot)=p$. Thus,  $\bar\mu_t=\int f_s\bar\mu_s \, dp(s)$ is constant for $p$-a.e.~$t\in T$. This implies that $\bar\mu$ is a product measure as required.     
\end{proof}

\subsection{Basin of attraction under conjugacy}
 To conclude the section, we study the relation between the basins of attraction of an $H$-invariant measure of the form $\nu \times \hat \mu$ and the basin of attraction of the $F$-invariant $\mathbb{P}$-Markovian measure $\bar\mu=\pi(\nu\times \hat\mu)$. We infer this conclusion from the following more general results. 

 In what follows, $Y$ and $\tilde{Y}$ denote measurable spaces and $g:Y \to Y$ and $\tilde{g}:\tilde{Y} \to \tilde{Y}$ measurable functions. We assume that there is a measurable function $\Pi:\tilde{Y} \to Y$ such that 
 $$g\circ\Pi = \Pi \circ \tilde{g}.$$  
 Denote by $\tilde{\mu}$ an ergodic $\tilde{g}$-invariant probability measure on $\tilde{Y}$ and $\mu=\Pi\tilde{\mu}$ its projection on~$Y$. For a bounded measurable function $\varphi:Y\to \mathbb{R}$, consider the basin of attraction 
$$
B(\mu,\varphi)=\bigg\{x\in Y : \lim_{n\to\infty} \frac{1}{n}\sum_{i=0}^{n-1} \varphi(g^i(x))=\int \varphi \, d\mu\bigg\}
$$
and
$$
B(\mu)=\bigg\{x\in Y : \lim_{n\to\infty} \frac{1}{n}\sum_{i=0}^{n-1} \delta_{g^i(x)}=\mu\bigg\},
$$
where the limit is taken in the weak$^*$ topology. Analogously, we introduce the basins $B(\tilde\mu,\psi)$   for a bounded measurable function $\psi:\tilde Y \to \mathbb{R}$ and $B(\tilde \mu)$. 

\begin{lem} \label{claim:varphi-basin} The probability 
    $\mu$ is an ergodic $g$-invariant measure and for every measurable bounded function $\varphi:Y \to \mathbb{R}$ it holds that $
    \Pi (B(\tilde\mu,\varphi\circ \Pi))\subset B(\mu,\varphi)$.
\end{lem}
\begin{proof}
    We have that $\mu$ is $g$-invariant since $g\mu=g\Pi\tilde\mu=\Pi\tilde g \tilde \mu = \Pi \tilde \mu = \mu$. Similarly,  $\mu$ is also ergodic. Indeed, if $A=g^{-1}(A)$, then $x\in \Pi^{-1}(A)=\Pi^{-1}(g^{-1}(A))$ if and only if $\Pi(\tilde g (x))=g(\Pi(x)) \in A $.  This is equivalent to $x \in \tilde g^{-1}(\Pi^{-1}(A))$. Consequently, $\tilde g^{-1}(\Pi^{-1}(A))=\Pi^{-1}(A)$. Now, $\mu(A)=\tilde{\mu}(\Pi^{-1}(A))\in \{0,1\}$ by the ergodicty of $\tilde\mu$.  

    On the other hand, let $x \in B(\tilde \mu, \varphi\circ\Pi)$. Then,
\begin{align*}
 \frac{1}{n} \sum _{i=0}^{n-1} \varphi ( g ^{i}(\Pi (x))) =
 \frac{1}{n} \sum _{i=0}^{n-1} \varphi ( \Pi ( \tilde g^i(x)))
\longrightarrow \int ( \varphi \circ \Pi ) \, d\tilde\mu 
=\int \varphi \, d\mu 
\end{align*}
as $n\to\infty$. 
Therefore, $\Pi (x) \in B(\mu, \varphi)$ and the desired inclusion is satisfied.  
\end{proof}

\begin{prop} \label{prop:basin}
    If $Y$ and $\tilde Y$ are both Polish spaces and $\tilde{\mu}(C_\Pi)=1$, where $C_\Pi$ are the continuity points of~$\Pi$, then 
    $\Pi(B(\tilde\mu))\subset B(\mu)$. 
\end{prop}
 \begin{proof} Since $Y$ is Polish space, we have a countable set $S$ of continuous bounded functions that determines the weak$^*$ topology on the set of probability measure of $Y$, \cite[Proposition~5.1]{BNNT22}. 
Let  $\varphi \in S$. By Lemma~\ref{claim:varphi-basin}, we have that $\Pi(B(\tilde \mu,\varphi \circ \Pi))  \subset B(\mu, \varphi)$. Then  
 \begin{equation} \label{eq:inclusion}
     \Pi \bigg(\bigcap_{\varphi \in S} B(\tilde \mu,\varphi \circ \Pi)\bigg) \subset \bigcap_{\varphi \in S} \Pi \bigg(B(\tilde  \mu,\varphi \circ \Pi)\bigg)  \subset  \bigcap_{\varphi \in S} B(\mu, \varphi) = B(\mu). 
 \end{equation}     
Now, since $\varphi$ is a bounded continuous function, $\varphi \circ \Pi$ is a bounded measurable function and its set of continuity points $C_{\varphi\circ \Pi}$ contains~$C_{\Pi}$. Thus, $\tilde\mu(C_{\varphi\circ \Pi})=1$. Since for any $x\in B(\tilde\mu)$,  
the sequence of probability measures $\tilde{\mu}_n(x)=\frac{1}{n}(\delta_{x}+\dots+\delta_{\tilde{g}^{n-1}(x)})$ converges to $\tilde{\mu}$ in the weak$^*$ topology, by Portmanteau theorem~\cite[Theorem~13.16]{Klenke2008}, we get that  
$$
 \frac{1}{n}\sum_{i=0}^{n-1} \varphi(\Pi(\tilde{g}^i(x))) = \int \varphi\circ \Pi \, d\tilde{\mu}_n(x) \to \int \varphi\circ \Pi \, d\tilde{\mu}  \quad \text{as $n\to \infty$}.  
$$
Hence, $B(\tilde\mu) \subset B(\tilde\mu,\varphi\circ \Pi)$ and therefore $\Pi(B(\tilde\mu)) \subset B(\mu)$ by~\eqref{eq:inclusion}. 
 \end{proof}

\begin{cor} \label{cor:basin} If $T$ is a finite set and $X$ is a Polish space, then $\pi(B(\nu\times \hat\mu)) \subset B(\bar\mu)$ for every $\hat P$-stationary measure $\hat\mu$ on $\hat X$,  where $\bar\mu=\pi(\nu\times \hat\mu)$.
\end{cor}
\begin{proof} We apply Proposition~\ref{prop:basin} with $\tilde Y=\Sigma \times \hat X$, $Y=\Omega\times X$, $\Pi=\pi$, $g=F$ and $\tilde g = H$, where $\Sigma=S^\mathbb{N}$ with $S=[0,1]$, $\hat X=T\times  X$ and $\Omega=T^\mathbb{N}$. We also have that $\tilde \mu=\nu \times \hat \mu$ and $\mu=\bar \mu$, where $\nu=\rho^\mathbb{N}$ with $\rho$ the normalized Lebesgue measure on $S$. By considering in $T$ the discrete topology and in $S$ the restricted topology of the standard topology on $\mathbb{R}$, we have that $\tilde Y$ and $Y$ are both Polish spaces. Thus, we only need to show that the set $C_\pi$  of continuity points of $\pi$ has full $(\nu\times \hat \mu)$-measure.

Let us consider the set $D_\pi = (\Sigma\times \hat X) \setminus C_\pi$ of discontinuity points of $\pi$. 
Recall that 
\[
\pi\colon \Sigma\times \hat X \to \Omega\times X, \quad \pi\bigl(\xi,(t,x)\bigr) = \Bigl( (g^i_{\xi}(t))_{i\geq 1},\, x\Bigr).
\]
The only potential discontinuities of \(\pi\) arise from the first coordinate map
since the second coordinate map is continuous. Moreover, since $T$ has discrete topology, $(\xi,(t,x))$ is a continuity point of the first coordinate map of $\pi$ if and only if $\xi$ is a continuity point of 
$$ G\colon \Sigma \to \Omega, \quad  G((\xi_0,\xi_1,\dots)) = \bigl(g_{\xi_0}(t),\, g_{\xi_1}\circ g_{\xi_0}(t),\,\dots
\bigr).
$$ 
Thus, to get that $(\nu\times \hat\mu)(D_\pi)=0$, it suffices to prove that the set of discontinuity points of \(G\) is a \(\nu\)-null set.

For each \( n\ge 1 \), define the finite coordinate map
\[
G_{n}\colon S^n \to T, \quad G_{n}(s_0,\dots,s_{n-1}) = g_{s_{n-1}}\circ \cdots \circ g_{s_0}(t).
\]
We claim that the discontinuity set
\[
D_n := \bigl\{ (s_0,\dots,s_{n-1})\in S^n : G_{n} \text{ is discontinuous at } (s_0,\dots,s_{n-1})\bigr\}.
\]
has Lebesgue measure zero. We prove this claim by induction.

For \(n=1\), we have that the discontinuities of \(G_{1}\), occur exactly on the set 
$$
D(t):= \big\{s'\in S:  \text{$s'$ is a discontinuty of $s\in S\mapsto g_s(t) \in T$ } \big\}.
$$
According to~\cite[Proposition~1.5]{levin2017markov}, written  
$T=\{t_1,\dots,t_m\}$, $F_{t,0}=0$ and $F_{t,k}=\sum_{i=1}^k Q(t,t_i)$ for any $t\in T$ and $k=1,\dots,m$, 
one has that  
$g_s(t)=t_k$ for  $F_{t,k-1}< s \leq F_{t,k}$.
Hence, we have that 
\(
\rho(D(t))=0.
\)
Thus, \(D_1\) has Lebesgue measure zero.

Now we  assume that for some \( n-1\ge 1 \) the discontinuity set \(
D_{n-1} \subset S^{n-1}
\)
has Lebesgue measure zero. Write
\(
G_{n}(s_0,\dots,s_{n-1}) = g_{s_{n-1}}( G_{n-1}(s_0,\dots,s_{n-2}) ).
\)
A discontinuity of \( G_{n} \) at \( (s_0,\dots,s_{n-1}) \) may arise in one of two ways:
\begin{enumerate}
    \item \emph{Inherited discontinuity:} If \((s_0,\dots,s_{n-2})\in D_{n-1}\), then \(G_{n-1}\) is discontinuous at \((s_0,\dots,s_{n-2})\). Hence, regardless of the value of \(s_{n-1}\), the composition is discontinuous. Denote by
    \[
    A:=\{ (s_0,\dots,s_{n-1})\in S^n : (s_0,\dots,s_{n-2})\in D_{n-1}\}.
    \]
    By Fubini's theorem, since \(D_{n-1}\) is a null set in \(S^{n-1}\), the cylinder \(A\) has Lebesgue measure zero in \(S^n\).
    
    \item \emph{New discontinuity in the last coordinate:} Suppose that \((s_0,\dots,s_{n-2})\notin D_{n-1}\). Then, $(s_0,\dots,s_{n-1})$ is a discontinuity of $G_{n}$ precisely when  \( s_{n-1}\in D(G_{n-1}(s_0,\dots,s_{n-2}))\). That is, this point belongs to 
    \[
   \qquad \qquad B:=\Bigl\{ (s_0,\dots,s_{n-1})\in S^n : (s_0,\dots,s_{n-2})\notin D_{n-1} \text{ and } s_{n-1}\in D\bigl(G_{n-1}(s_0,\dots,s_{n-2})\bigr) \Bigr\}.
    \]
    Again, since $D(t)$ has null  Lebesgue measure in \(S\) for any $t\in T$, 
    by Fubini's theorem, \(B\) has Lebesgue measure zero.
\end{enumerate}
Thus, the full discontinuity set of \(G_{n}\) satisfies
\(
D_n \subset A \cup B,
\)
and therefore has zero Lebesgue measure. This completes the inductive proof.

\medskip
Now, observe that the map $G$ can be regarded as the limit (in the product topology) of the sequence of maps \(G_{n}\). Its discontinuities occur if and only if there exists some finite \(n\ge 1\) for which the restriction of \(G\) to the first \(n\) coordinates is discontinuous. In other words, if we define the cylinder sets
\[
\widetilde{D}_n := \Bigl\{ (s_0,s_1,\dots) \in \Sigma : (s_0,\dots,s_{n-1})\in D_n \Bigr\},
\]
then the set of discontinuity points of \(G\) is contained in
\[
D_\infty := \bigcup_{n\ge 1} \widetilde{D}_n.
\]
Since each \(\widetilde{D}_n\) has \(\nu\)-measure zero, it follows that \(D_\infty\) is a null set and thus the set of discontinuity points of \(G\) has zero \(\nu\)-measure. This completes the proof.
\end{proof}

\section{Mean constrictive iid representations}
\label{sec:mean-constrictive idd respresentations}

In this section, we prove Theorem~\ref{mainthmA} and Corollaries~\ref{maincorA} and~\ref{maincor:D}. To do this, we first obtain a more general result assuming that the annealed Perron--Frobenius operator associated with the iid representation of the Markovian random map is mean constrictive. In what follows, we introduce the necessary definition and framework to state and prove more general result stated in Theorem~\ref{thm:mainthm-MC-FPM}. 

\subsection{Mean constrictive}

Let $(X , \mathscr B,m)$ and $(T, \mathscr A,p)$ be {standard Borel probability spaces}, where $p$ is the unique stationary measure of a transition probability $Q(t,A)$ in $T\times \mathscr{A}$ and consider its associated Markov measure $\mathbb{P}$ on $(\Omega,\mathscr{F})=(T^\mathbb{N},\mathscr{A}^\mathbb{N})$. As usual, we denote $\hat X = T \times X$, $\hat m= p\times m$ and $\bar{m}=\mathbb{P}\times m$. 

Let  $f:T\times X\to X$ be a random map driven by the Markov measure $\mathbb{P}$ and consider its iid representation $h:S\times \hat X \to \hat X$   introduced in~\eqref{eq:idd-representation}. Recall that $(S,\rho)$ is the unit interval equipped with the Lebesgue measure, $(\Sigma,\nu)=(S^\mathbb{N},\rho^\mathbb{N})$ and $H:\Sigma \times \hat X \to \Sigma \times \hat X$ denotes the associated skew-product with the random map $h$.

\begin{lem} \label{lem:nonsingular-H} If $f$ is fibered $(p,m)$-nonsingular, then $H$ is $(\nu\times \hat m)$-nonsingular. 
\end{lem}
\begin{proof}  Let $A$ be a measurable subset of $\Sigma \times \hat X$  such that $(\nu \times \hat m) (A)=0$. For each $(\xi,t)\in \Sigma \times T$, denote   $A_{\xi,t}=\{x\in X: (\xi,(t,x))\in A\}$. Hence, since 
$$
0=(\nu \times \hat m)(A)= \int m(A_{\xi,t})\, dp(t) \,d\nu(\xi),
$$
there is a set $B_A \subset \Sigma \times T$ with $(\nu \times p)(B_A)=1$ such that $m(A_{\xi,t})=0$ for every $(\xi,t)\in B_A$.  Moreover, since $f$ is fibered $(p,m)$-nonsingular, there is $T_0 \subset T$ with $p(T_0)=1$ such that $f_tm \ll m$. Thus, $N=B_A\cap (\Sigma\times T_0)$ has $(\nu\times p)(N)=1$ and $f_tm(A_{\xi,t})=0$ for every $(\xi,t)\in N$. 

On the other hand, $(\zeta,(r,y))\in H^{-1}(A)$ if and only if 
$$
(\sigma(\zeta),(g^{}_{\zeta_0}(r),f_{g^{}_{\zeta_0}(r)}(y)))=
(\sigma(\zeta),h_{\zeta_0}(r,y))=H(\zeta,(r,y))\in A.
$$
This is equivalent to 
$$
f_{g^{}_{\zeta_0}(r)}(y))\in A_{\sigma(\zeta),g^{}_{\zeta_0}(r)}.
$$
Then, using that $\nu=\rho^\mathbb{N}$, and the transition probability $Q(r,\cdot)=\rho(\{s\in S: g_s(r)\in \cdot \})=g^r\rho$, where $g^r=g(\cdot,r)$, we get that
\begin{align*}
    H(\nu\times \hat m)(A) &= (\nu\times \hat m)(H^{-1}(A))=\int m((f_{g^{}_{\zeta_0}(r)})^{-1}( A_{\sigma(\zeta),g^{}_{\zeta_0}(r)})) \, dp(r)\, d\nu(\zeta) \\
    &=\int f_{g^{}_{\zeta_0}(r)}m( A_{\xi,g^{}_{\zeta_0}(r)}) \, dp(r)\, d\rho(\zeta_0) \, d\nu(\xi)  =\int f_{t}m( A_{\xi,t}) \,  dg^r\rho(t) \,dp(r)\, d\nu(\xi) \\
    &= \int f_{t}m( A_{\xi,t}) \,  Q(r,dt) \,dp(r)\, d\nu(\xi) = \int f_{t}m( A_{\xi,t}) \,  dp(t)\, d\nu(\xi).
\end{align*}
The last equality follows because $p$ is a $Q$-stationary measure, that is, $p=\int Q(r,\cdot)\, dp(r)$, and thus $dp(t)=Q(r,dt)\,dp(r)$. Since 
$$
\int f_{t}m( A_{\xi,t}) \,  dp(t)\, d\nu(\xi) = \int_{N} f_{t}m( A_{\xi,t}) \,  dp(t)\, d\nu(\xi)=0,
$$
we conclude that $H(\nu\times \hat m)(A)=0$. This shows that $H$ is $(\nu\times \hat m)$-nonsingular.  
\end{proof}

From Lemmas~\ref{lem:nonsingular-H},~\ref{lem:nonsigular-relations}, \ref{lem:f-nonsigular-P} and Remark~\ref{rem:def-Perron},  under the assumption that $f$ is fibered $(p,m)$-nonsingular, the annealed Perron--Frobenius operator $\mathcal{L}_h:L^1(\hat m) \to L^1(\hat m)$ is well-defined.

\begin{rem} 
A similar argument shows that if $F$ is $\bar m$-nonsingular, then $h$ is $(\rho,\hat m)$-nonsingular. Then, the annealed Perron--Frobenius operator $\mathcal L_h$ associated with $h$ is also well-defined. However, this is not enough to apply Theorem~\ref{prop:iid}, which requires that $H$ is $(\nu\times \hat m)$-nonsingular.   
\end{rem}

\begin{thm} \label{thm:mainthm-MC-FPM}
If $f$ is fibered $(p,m)$-nonsingular and  
$\mathcal L_h$ is {\tt (MC)}, then the associated skew-product $F$ satisfies {\tt (FPM)} with respect to $\bar m$.
\end{thm}

\begin{proof} 
According to Theorems~\ref{thm:BNNT} and \ref{prop:iid}, since $\mathcal L_h$ is mean constrictive, then we have $H$ satisfies $\mathtt{(FPM)}$ with respect to $\nu\times \hat m$. That is, there exist finitely many ergodic $H$-invariant probability measures $\nu \times \hat\mu_1,\dots, \nu\times \hat\mu_r$ on $\Sigma\times \hat X$  such that they are absolutely continuous with respect to $\nu\times\hat m$ and have pairwise disjoint supports (up to a $(\nu \times \hat m)$-null set).  Furthermore,
$$(\nu  \times \hat m)\big(B(\nu\times \hat\mu_1,\psi) \cup \dots \cup B(\nu \times \hat \mu_r,\psi)\big)=1$$
for any measurable  bounded function $\psi:\Sigma\times \hat X \to \mathbb{R}$, where
\[
B(\nu\times \hat\mu_i,\psi)=\left\{
(\xi,(t,x))\in \Sigma \times \hat X : \lim_{n\to\infty}\frac{1}{n}\sum_{j=0}^{n-1}\psi\big(H^j(\xi,(t,x))\big)=\int\psi\, d(\nu\times\hat\mu_i)
\right\}.
\]

Let $\pi: \Sigma \times \hat X \to \Omega \times X$ be the measurable map given in Lemma~\ref{lem:matias}. Since $\nu \times \hat \mu_i$ is an ergodic $H$-invariant probability measure on $\Sigma \times \hat X$, we get that $\hat \mu_i$ is an ergodic $h$-stationary measure on $\hat X$. Then, according to Theorem~\ref{thm:Matias} and Remark~\ref{eq:mu-hat-bar}, we have
\begin{enumerate}[leftmargin=0.5cm]
    \item[1)]  
 $\bar \mu_i  = \pi(\nu\times \hat\mu_i)$ is an   ergodic $F$-invariant $\mathbb{P}$-Markovian measure for every $i=1,\dots,r$;
\item[2)] Absolutely continuity: 
Consider a set $B \subset \Omega \times X$ with $(\mathbb{P}\times m)({B})=0$.  
By Lemma~\ref{lem:matias}, 
$$
(\nu \times \hat{m})(\pi^{-1}(B))=\pi(\nu \times \hat{m})(B)=(\mathbb{P}\times m)(B)=0.
$$
Hence, it follows from $(\nu\times\hat\mu_i)\ll \nu \times \hat m$ that 
$0=(\nu\times\hat\mu_i)(\pi^{-1}(B))=\pi(\nu\times\hat\mu_i)(B)=\bar{\mu}_i(B)$. 
This shows that $\bar{\mu}_i\ll \bar{m}=\mathbb{P}\times m$.

\item[3)] Union of the basin of attraction: Let $\psi:\Omega \times X \to \mathbb{R}$ be a measurable bounded function. 

Now, by Lemma~\ref{claim:varphi-basin}, we get
\[
 1=(\nu \times \hat m) \left(\bigcup _{i=1}^r B(\nu \times \hat\mu_i, \psi\circ\pi)\right) \leq (\nu \times \hat m) \left(\bigcup _{i=1}^r \pi^{-1}(B( \bar\mu_i, \psi))\right) = 
(\mathbb{P}\times m)\left(\bigcup _{i=1}^r B (\bar \mu _i, \psi)\right).
\]
This implies that $\bar m( B(\bar \mu _1, \psi) \cup \dots \cup  B(\bar \mu _r, \psi))=1$ as desired. \qedhere
\end{enumerate}
\end{proof}

\subsection{Proof of Theorem~\ref{mainthmA}}  By Lemma~\ref{lem:nonsingular-H}, since $f$ is $(p,m)$-nonsingular, the skew-product $H$ associated with the iid representation $h$ of $f$ is $(\nu\times \hat m)$-nonsingular. Hence, according to Lemmas~\ref{lem:nonsigular-relations} and~\ref{lem:f-nonsigular-P} and Remark~\ref{rem:def-Perron}, $\mathcal{L}_h:L^1(\hat m) \to L^1(\hat m)$ is a well-defined Markov operator and its dual operator $\mathcal{L}^*_h:L^\infty(\hat m) \to L^\infty(\hat m)$ coincides with the operator $\hat{P}$ introduced in the statement of Theorem~\ref{mainthmA}. Moreover, 
$$\mathcal{L}_h^*1_B(t,x)=\hat P((t,x),B), \quad \text{for every measurable set $B$ in $\hat X$ and $\hat m$-a.e.~$(t,x)\in \hat X$.}$$
Since by assumption, both $T$ and $X$ (and thus $\hat X)$ are compact Polish spaces and $\hat P^{n_0}((t,x),B)$ is also strong Feller, we get from~\cite[Theorems~B and~2.3]{BNNT22} that $\mathcal{L}^{2n_0}_h$ (and thus also $\mathcal{L}^{2n_0*}_h=\hat P^{2n_0}$) is a compact operator.  
In particular, $\mathcal{L}_h$ is quasi-compact and by~\cite[Proposition 6.1 and Figure~1]{BNNT22} is {\tt (MC)}. Consequently, by Theorem~\ref{thm:mainthm-MC-FPM}, we also get that $F$ satisfies {\tt (FPM)} with respect to $\bar m= \mathbb{P}\times m$. This concludes the proof of the theorem.    

\subsection{Proof of Corollary~\ref{maincorA}} \label{ss:proof-corollary-D}

 First, we will introduce the following clean definition of the notion of physical noise: 

\begin{defi} \label{def-physical-noise}  A Bernoulli probability $\mathbb{P}=p^\mathbb{N}$ on $\Omega=T^\mathbb{N}$ (or the initial distribution $p$ on $T$) is said to be \emph{physical noise} for a random map $f:T\times X \to X$ driven by $\mathbb{P}$  and with reference measure $m$ on $X$ and transition probability $P(x,A)$ 
if one of the following equivalent statements holds: 
\begin{enumerate}[label=(\roman*)]
    \item  there exists $n_0\geq 1$ such that 
    $P^{n_0}(x,A)$ is Feller continuous and $P^{n_0}(x,\cdot)\ll m$ for every $x\in X$;
\item  there exists $n_0\geq 1$ such that 
$P^{n_0}(x,dy)=q(x,y) \, dm(y)$ with $x\mapsto q(x,\cdot) \in L^1(m)$ continuous. 
\end{enumerate}
\end{defi}

The equivalence between (i)--(ii) was proved in~\cite{BNNT22}. Namely, according to \cite[Theorem~2.6 and Remark~2.1]{BNNT22}, (i) implies that $P^{2n_0}(x,A)$ is ultra Feller continuous and $P^{2n_0}(x,\cdot)\ll m$ for every $x\in X$.  This is equivalent to (ii), cf.~\cite[Proof of Proposition~2.2]{BNNT22}. Moreover, we also have the following:

\begin{rem}  If $m$ is fully supported on $X$, according to \cite[Theorem~2.6 and Proposition~2.2]{BNNT22}, physical noise is also equivalent to
\begin{enumerate}
    \item[(iii)] there exists $n_0\geq 1$ such that 
$P^{n_0}(x,A)$ is $m$-nonsingular and strong Feller.
\end{enumerate}
    
\end{rem}

Now we will complete the proof of Corollary~\ref {maincorA} as follows.  Let $(X,\mathcal{B},m)$, $(T,\mathscr{A},p)$ and $f:T\times X \to X$ be as in the statement of Corollary~\ref{maincorA}. Now we consider the iid representation $h:S \times \hat{X} \to \hat{X}$ of the random map $f$ given in~\eqref{eq:idd-representation}. Recall that here $S=[0,1]$ endowed with the normalized Lebesgue measure $\rho$. As observed in~\eqref{eq:P-h}, $\hat P((t,x),B)$ is the transition probability associated with $h$. Since, by assumption, $\hat P^{n_0}((t,x),B)$   is Feller continuous and $\hat P((t,x),\cdot)\ll p\times m$ for all $(t,x)\in \hat T\times X=\hat{X}$, we see that $\nu=\rho^\mathbb{N}$ is a physical noise for $h$.  
 Hence, it follows from~\cite[Proposition~2.5 and Theorem~2.6]{BNNT22} that $\hat{P}^{n_0}((t,x),B)$ is strong Feller continuous. Therefore, we get Corollary~\ref{maincorA} from Theorem~\ref{mainthmA}.  

The following remark provides reasonable sufficient conditions to get the required Feller continuity in Corollary~\ref{maincorA}.

\begin{rem} As a consequence of the results of Blumenthal and Corson~\cite{BC70,BC72}, if $Q(t,A)$ is Feller continuous and either
\begin{enumerate}
    \item $T$ totally disconnected or
    \item $T$ connected, locally connected and $Q(t,\cdot)$ has full support for each $t\in T$,  
\end{enumerate}
then $Q(t,A)$ can be represented as the transition probability associated with continuous maps. That is,
there exists an auxiliary parametric space $(S,\mathscr{C},\rho)$  and a random map $g:S\times T \to T$, where $g_s=g(s,\cdot)$ is a continuous map for every $s\in S$ such that 
 $$
 Q(t,A)=\rho(\{s\in S: \, g_s(t)\in A\}), \quad \text{for every} \ t\in T, \ \ A\in \mathscr{A}.
 $$
See also~\cite[Section~1.1]{kifer-ergodic}. If in addition  $f:T\times X \to X$ is also a continuous map, we get that $h_s: \hat X \to \hat X$ given by $h_s(t,x)=(g_s(t),f_{g_s(t)}(x))$ is also a continuous map for every $s\in S$. In particular, $h:S\times \hat X \to \hat X$ is a continuous random map. Thus, under the above assumptions, we get that $\hat{P}((t,x),B)$ is Feller continuous. 
\end{rem}

\subsection{Proof of Corollary~\ref{maincor:D}}

We first give a general sufficient condition for Corollary~\ref{maincorA} to hold
and then apply it to random maps with additive noise.

\begin{thm}\label{thm:thm proof cor D}
Let $X = T$ and $S$ be compact Riemannian manifolds equipped with their normalized Lebesgue measure $m=p$ and $\rho$ respectively. 
Consider  $C^1$ maps $f: T \times X \to X$  and $g: S \times T \to T$ and a transition probability $Q(t,A)$ given by $Q(t,A)=\rho (\{ s\in S : g_s(t)\in A\})$. 
Assume that $p$ is the unique $Q$-stationary probability measure. Also we require that 
there exists $n_0\geq 1$ such that 
$$
 H_{(t,x)}: S^{n_0} \to T\times X,  \quad H_{(t,x)}(s_1,\dots,s_{n_0})=h_{s_{n_0}}\circ \dots \circ h_{s_1}(t,x)
$$
is a submersion, where 
$$h_s(t,x)=(g_s(t),f_{g_s(t)}(x)), \quad (s,t,x)\in S\times T \times X.$$
Then, the transition probability $\hat P((t,x),B)$ is Feller continuous and $P^{n_0}((t,x),B)\ll p\times m$ for every $(t,x)\in T\times X$. Therefore, the conclusions of Theorem~\ref{mainthmA} hold. 
\end{thm}
 
\begin{proof}
Since $Q(t,A)$ is the transition probability associated with the Bernouilli random map $g:S\times T \to T$, we have, similarly to before, that $\hat P((t,x),B)$ is the transition probability associated with the random map $h: S \times \hat X \to \hat X$ given by $h_s(t,x)=(g_s(t),f_{g_{s}(t)}(x))$, $(t,x)\in T\times X=\hat X$, $s\in S$. Also, since $f$ and $g$ are $C^1$, we get that $h$ is continuous and then, $\hat P((t,x),B)$ is Feller continuous. Moreover, by~\cite[Proposition 3.2]{benaim2024invariant},  since for some $n_0\geq 1$ the function $H_{(t,x)}:S^{n_0} \to \hat X$ is a submersion for every $(t,x)\in \hat X$, 
then $\hat{P}^2((t,x), \cdot)$ is absolutely continuous with respect to $\hat m=p\times m$ 
for all $(t,x) \in \hat X$. Then, we can apply Corollary~\ref{maincorA} and obtain the desired result.
\end{proof}

\begin{proof}[Proof of Corollary~\ref{maincor:D}]
We show that the example given by $g(s,t)=s+t$ and $f(t,x)=f_0(x)+t$ with $S=T$ and $p=\rho$, where $f_0$ is a diffeomorphism, satisfies all the conditions of Theorem~\ref{thm:thm proof cor D}. 

First, since the normalized Lebesgue measure (the Haar measure) $\rho$ is the unique $g$-stationary measure on $T$, we have that the transition probability $Q(t,A)$ has $p=\rho$ as the unique $Q$-stationary measure. On the other hand, 
$$
H_{(t,x)}(s_1,s_2)=(s_1+s_2+t, f_0(f_0(x)+s_1+t)+s_1+s_2+t).
$$
Then, the determinant of 
\[ DH_{(t,x)}(s_1,s_2)= \begin{pmatrix} 1 & 1 \\ Df_0(f_0(x)+s_1+t)+1 & 1 \end{pmatrix}\]
 is different from zero, and $f_0$ is a diffeomorphism. Thus, we see that $H_{(t,x)}$ is a submersion for all $(t,x)\in \hat X$, concluding that the example satisfies all required assumptions.  
\end{proof}

\section{Mostly contracting Markovian random maps: Proof of Theorem~\ref{cor:Markovian-qc}} \label{sec:Markov}

Let $(X,d)$ denote a compact metric space and consider  $(\Omega,\mathscr{F})= (T^\mathbb{N},\mathscr{A}^\mathbb{N})$, where $(T,\mathscr{A})$ is a measurable space. Let $\mathbb{P}$ be an ergodic shift-invariant Markov measure on~$\Omega$. As usual, denote $\hat X = T \times X$.

Let $f:T\times X \to X$ be a random map driven by the Markov measure $\mathbb{P}$ and consider its iid representation  $h:S\times \hat X \to \hat X$ introduced in~\eqref{eq:idd-representation}, where $(S,\rho)$ is the unit interval equipped with the Lebesgue measure and $(\Sigma,\nu)=(S^\mathbb{N},\rho^\mathbb{N})$. Now, we will proceed to analyze the maximal Lyapunov exponent of $h$ and compare it with that of $f$. First, we need to study the local and global Lipschitz constants of $h_s$.  To do this, we endow $T$ with the discrete metric $\delta$, i.e.~$\delta(t,t')=1$ if $t=t'$ and $0$ otherwise, and define a metric $\hat{d}$ on $\hat X$ as
$$\hat{d}((t,x),(t,x'))= \delta(t,t')+d(x,x').$$ 

\begin{lem} \label{lem:Mark-lip-iid}
    If $f$ is a Lipschitz random map such that $\sup_{t\in T} \mathrm{Lip}(f_t)<\infty$,  then $h$ is a Lipschitz random map. Moreover, for any $n\geq 1$,
    \begin{equation} \label{eq:cont-lip}
Lh^n_\xi(t,x)=Lf^n_\omega(x)  \quad   \text{where \   $\xi \in \Sigma$,  $(t,x)\in\hat X$ and $\omega=(g^k_\xi(t))_{k\geq 1} \in \Omega$}.
\end{equation}
and
\begin{equation} \label{eq:lip}
\mathrm{Lip}(h_s)\leq M\eqdef \max\big\{ \sup_{t\in T} \mathrm{Lip}(f_t),1+\diam(X), 2\big\}<\infty \quad \text{for all $s \in S$}.
\end{equation}
\end{lem}
\begin{proof}
 We have that
$\hat{d}((t,x),(t',x'))<1$ if and only if $t=t'$ and $d(x,x')<1$. In particular, for any $(t,x)\in \hat X$ and $0<r<1$, if
$(t',x'),(t'',x'')\in B((t,x),r)$, 
$$d((t',x'),(t'',x''))=d(x',x'') \quad \text{and} \quad d(h_s(t',x'),h_s(t'',x''))=d(f_{g_s(t)}(x'),f_{g_s(t)}(x'))$$
for all $s\in S$. Thus, $$\mathrm{Lip}(h_s|_{B((t,x),r)}) =\mathrm{Lip}(f_{g_s(t)}|_{B(x,r)}) \quad \text{for $0<r<1$}$$ 
and consequently, $Lh_s(t,x)=Lf_{g_s(t)}(x)$ for all  $s\in S$ and $(t,x)\in \hat  X$. Moreover, by a similar argument, since for any $n\geq 1$,  
$$
 h^n_\xi(t,x)=(g^n_\xi(t),f_\omega^n(x)), \quad \text{where  $\xi \in S^\mathbb{N}$,  $(t,x)\in \hat X$ and $\omega=(g^k_\xi(t))_{k\geq 1} \in \Omega$,} 
$$
we also get that~\eqref{eq:cont-lip}. 
Moreover, by a direct calculation, 
setting $r=\hat{d}((t,x),(t',x'))$ and $L=\sup \{\mathrm{Lip}(f_t): t\in T\}$, we have that either,
$$
\hat{d}(h_s(t,x),h_s(t',x')) \leq  L \, d(x,x') = L \cdot 
r \quad \text{if $r<1$}
$$
 or
$$
\hat{d}(h_s(t,x),h_s(t',x')) \leq (1+\diam(X)) \cdot r \quad \text{if $r\geq 1$ and $d(f_{g_s(t)}(x),f_{g_s(t')}(x'))\geq 1$}
$$
or
$$
\hat{d}(h_s(t,x),h_s(t',x')) \leq  2 r \quad \text{if $r\geq 1$ and $d(f_{g_s(t)}(x),f_{g_s(t')}(x'))< 1$.}
$$
This shows~\eqref{eq:lip} and thus $h_s$ is Lipschitz for all $s\in S$. The proof of the lemma is completed. 
\end{proof}

We are now ready to compare the maximal Lyapunov exponent. First, recall that $\pi:\Sigma \times \hat X \to \Omega \times X$ denotes the map given in Lemma~\ref{lem:matias} that conjugates the skew-products $F$ and $H$ associated with $f$ and $h$, respectively. 

\begin{prop} \label{prop:Lyapunov}
If $\hat{\mu}$ is an $h$-stationary probability measure, then 
$$
\lambda(\hat{\mu})\eqdef\lambda(\nu\times\hat{\mu})=\lambda(\bar{\mu}), \quad  \text{where \ $\bar{\mu}=\pi(\nu\times\hat\mu)$}.
$$
\end{prop}
\begin{proof} From~\eqref{eq:cont-lip} and the notation introduced in~\S\ref{sec:lyapunov-def}, we have that,
$$
Lh_\xi^n(t,x) = L_n(\pi(\xi,(t,x))) \quad \text{for all $n\geq 1$,}
$$
and thus,
\begin{align*}
    \lambda(\nu\times\hat{\mu}) &= \lim_{n\to\infty} \frac{1}{n} \int \log Lh_\xi^n(t,x)\, d(\nu\times\hat{\mu}) = \lim_{n\to\infty} \frac{1}{n} \int \log L_n(\pi(\xi,(t,x))) \, d(\nu\times\hat{\mu}) \\ 
    &=   \lim_{n\to\infty} \frac{1}{n} \int \log L_n(\omega, x) \, d\pi(\nu\times\hat{\mu}) = \lim_{n\to\infty} \frac{1}{n} \int \log L_n(\omega,x) \, d\bar{\mu} \\ 
    &=\lim_{n\to\infty} \frac{1}{n} \int \log Lf^n_\omega(x) \, d\bar{\mu}
    =\lambda(\bar{\mu}).  \qedhere
\end{align*}
\end{proof}

\subsection{Proof of Theorem~\ref{cor:Markovian-qc}}

We first note that the hypothesis in Theorem~\ref{cor:Markovian-qc} implies the standing assumption in Section~\ref{sec:iid-representation}.
If $\mathbb{P}$ is an ergodic Markov measure on $\Omega$, then we have a transition probability $Q(t,A)$ and an ergodic $Q$-stationary measure $p$ as the initial distribution of $\mathbb{P}$.
Since $T$ is finite, we can modify $Q(t,A)$ by $\tilde{Q}(t,A)$ which equals zero outside the support of $p$, and hence $p$ will be the unique $\tilde{Q}$-stationary measure.
In other words, under the assumption that $T$ is finite, any ergodic Markov measure $\mathbb{P}$ on $\Omega=T^\mathbb{N}$ is obtained from a uniquely ergodic transition probability on $T$, which is our standing assumption.

Since, by assumption of Theorem~\ref{cor:Markovian-qc},  $T$ is a finite set and $f$ is a Lipschitz random map, by Lemma~\ref{lem:Mark-lip-iid}, we have that the iid representation $h$ of the Markovian iteration of $f$ is a Lipschitz random map. As~\eqref{eq:lip} provides a uniform bound for the Lipschitz contact of $h_s$, the corresponding integrability condition~\eqref{eq:integrability} and the exponential moment condition
\begin{equation}\label{eq:integral_condition}
  \int \mathrm{Lip}(h_{\omega_0})^\beta \, d\mathbb{P}(\omega) < \infty \quad \text{for some $\beta > 0$}
\end{equation}
(indeed, for any $\beta>0$) are fulfilled. Moreover, given an $h$-stationary measure $\hat{\mu}$ on $\hat  X$, Proposition~\ref{prop:Lyapunov} implies that 
$\lambda(\hat{\mu})=\lambda(\bar{\mu})$, where  $\bar{\mu}=\pi(\nu\times \hat{\mu})$ is an $F$-invariant $\mathbb{P}$-Markovian invariant measure according to Theorem~\ref{thm:Matias} and Remark~\ref{eq:mu-hat-bar}.
Hence, since $f$ is by assumption $\mathbb{P}$-mostly contracting, it holds that $\lambda(\bar{\mu})<0$ and thus $\lambda(\hat{\mu})<0$. Therefore, $h$ is a mostly contracting Bernoulli random map.

Since $T$ is finite, we have that $(\hat X,d)$ is a compact metric space and then, since $h$ is a mostly contracting random map satisfying the exponential moment condition~\eqref{eq:integral_condition}, we can apply~\cite[Theorem~B]{BM24}.  We get that for any $\alpha>0$ small enough, the annealed Koopman operator  $P:C^\alpha(\hat X) \to C^\alpha(\hat X)$  defined by 
$$
    P\varphi (t,x) = \int \varphi \circ h_s (t,x)\, d\rho(\xi) \qquad \varphi \in C^\alpha (\hat X) 
$$
is quasi-compact. Since by~\eqref{eq:P-h} we have that $\hat P((t,x),d\hat u)=\rho(\{s\in S: h_s(t,x)\in B\})=h^{(t,x)}\rho$, where $h^{(t,x)}=h(\cdot,(t,x))$, then
$$
\int \varphi \circ h_s (t,x)\, d\rho(s)  = \int \varphi(\hat u) \, dh^{(t,x)}\rho(\hat u) \int \varphi(\hat u)  \, \hat P((t,x),d\hat u) = \hat P \varphi (t,x), 
$$
and thus $P$ coincides with the operator $\hat P$ 
associated with the transition probability $\hat{P}((t,x),B)$.  To get {\tt (fpm)} for the skew-product $F$ with respect to $\bar m= \mathbb{P}\times m$ we need some extra work.

First, by~\cite[Corollary~I(b)]{BM24}, given a probability measure $m$ on $X$, we get that finitely many ergodic $h$-stationary probability measures $\hat{\mu}_1,\dots,\hat{\mu}_r$ on $\hat X$ such that 
$$
 (\nu\times \hat m) \left( B(\nu\times \hat{\mu}_1) \cup \dots \cup B(\nu\times \hat{\mu}_r)\right)=1  \quad \text{and} \quad (\nu\times \hat m)(B(\nu\times \hat{\mu}_i))>0  \ \text{for $i=1,\dots,r$},
$$
where $\hat m =p\times m$. Again, by Theorem~\ref{thm:Matias} and Remark~\ref{eq:mu-hat-bar}, $\bar{\mu}_i=\pi(\nu\times \hat{\mu}_i)$ is an ergodic $F$-invariant $\mathbb{P}$-Markovian measure on $\Omega\times X$ for all $i=1,\dots,r$. Moreover, from Lemma~\ref{lem:matias}, we have that $\pi(\nu\times \hat m)=\mathbb{P}\times m =\bar{m}$. Now, by Corollary~\ref{cor:basin}, $\pi(B(\nu\times \nu_i)) \subset B(\bar\mu_i)$ for $i=1,\dots,r$. 
As a consequence,
$$
 \bar{m}(B(\bar{\mu}_i))=\pi(\nu\times \hat m)(B(\bar{\mu}_i))=(\nu\times \hat m)(\pi^{-1}(B(\bar{\mu}_i))\geq (\nu\times \hat m)(B(\nu\times \hat \mu_i))>0 \quad \text{for $i=1,\dots,r$}
$$
and
\begin{align*}
\bar{m}\left(B(\bar{\mu}_1)\cup\dots\cup B(\bar{\mu}_r)\right) &=
\pi(\nu\times \hat m)\left(B(\bar{\mu}_1)\cup\dots\cup B(\bar{\mu}_r)\right) \\ 
 &\geq (\nu\times \hat m) \left( B(\nu\times \hat{\mu}_1)\cup \dots \cup B(\nu\times \hat{\mu}_1)\right)=1.
\end{align*}
This concludes the proof of Theorem~\ref{cor:Markovian-qc}.

\section{Markovian invariant principle: Proof of Corollaries~\ref{cor:mostly-contraction-circle} and~\ref{cor:interval}} \label{sec:invariant}
An invariance principle is a statement on the rigidity of the invariant measure, given the condition that the extremal Lyapunov exponents coincide. In one dimension, this condition is equivalent to the maximal Lyapunov exponent vanishing. The pioneering establishment of this principle appears in the seminal work by Ledrapier~\cite{Le86}, in the context of linear cocycles (random sequences of matrices).  The nonlinear smooth version of this principle was first established by Crauel in~\cite{crauel:1990} (see also~\cite{AV10} and~\cite{Mal:17}). Here we get the following invariant principle for Markovian random maps: 

\begin{thm} \label{thm:Markovian-inv-principle}
Let $(T,\mathscr{A})$ be a standard Borel space and consider an ergodic shift-invariant Markov measure $\mathbb{P}$ on $(\Omega,\mathscr{F})=(T^\mathbb{N},\mathscr{A}^\mathbb{N})$. Let $F:\Omega \times \mathbb{S}^1 \to \Omega \times \mathbb{S}^1$ be the skew-product associated with  a random map $f:T\times  \mathbb{S}^1\to \mathbb{S}^1$, where $f_{\omega_0}$ is a  $C^1$ diffeomorphism  
for $\mathbb{P}$-a.e.~$\omega\in\Omega$ and~\eqref{eq:integrability} holds. Then, for every ergodic  $F$-stationary $\mathbb{P}$-Markovian measure $\bar{\mu}$ on $\Omega\times \mathbb{S}^1$, one of the following conditions holds:
    \begin{itemize}
        \item contraction: $\lambda(\bar{\mu})<0$, or
        \item invariance: $f_{\omega_0}\mu_{\omega_0}=\mu_{\omega_1}$  for $\mathbb{P}$-a.e.~$\omega=(\omega_i)_{i\geq 0}\in \Omega$, where $d\bar\mu=\mu_{\omega_0}\, d\mathbb{P}(\omega)$. 
    \end{itemize}
\end{thm}

\begin{proof}  Assume that $\lambda(\bar\mu)\geq 0$. We can apply~\cite[Proposition~7.2]{BM24} to get that $\lambda_{con}(\bar{\mu})=0$, where \begin{align*}
\lambda_{con}(\bar{\mu}) = \int \lambda_{con}(\omega,x)\, d\bar{\mu}, \quad \text{with}  \quad
\lambda_{con}(\omega,x) \eqdef \limsup_{y \to x} \limsup_{n \to \infty} \frac{\log(d(f_\omega^n(x), f_\omega^n(y)))}{n}.
\end{align*} 
Then, according to~\cite[Theorem~F]{Mal:17} (see also~\cite[Theorem~7.1]{BM24}), and using the fact that the probability $\bar\mu$ is a $\mathbb{P}$-Markovian measure, the disintegration $d\bar{\mu}=\mu_{\omega_0}\, d\mathbb{P}(\omega)$ satisfies that 
\begin{equation*} \label{eq:invariance2}
   \mu_{\omega_1}=f_{\omega_0}\mu_{\omega_0}  \quad \text{for $\mathbb{P}$-a.e.~$\omega=(\omega_i)_{i\geq 0}\in \Omega$.} 
\end{equation*}
This concludes the proof of the theorem.
\end{proof}

\begin{rem} In~\cite[proof of Theorem~5.1, after Equation~(5.2)]{matias:2022}, 
the author claims that if $f$ does not have an invariant measure in common, then the invariance condition in the sense of Theorem~\ref{thm:Markovian-inv-principle} does not hold.
This is not true in general (for instance, take $\mathbb{P}$ periodic).
This affects~\cite[Theorem~2.4]{matias:2022}, whose conclusion only holds under the additional assumption that the invariance condition in the sense of Theorem~\ref{thm:Markovian-inv-principle} does not hold.  
\end{rem}

\subsection{Proof of Corollary~\ref{cor:mostly-contraction-circle}} 
 Let $\bar{\mu}$ be a $f$-stationary probability measure on $\Omega\times X$.  It is well established that the set of $\mathbb{P}$-Markovian $F$-invariant  measures, when it is nonempty, is convex with a nonempty set of extreme points. By the ergodic decomposition theorem~\cite[Appendix A, Theorem~1.1]{kifer-ergodic}, we can write 
$\bar{\mu}(E)=\int \bar{\nu}(E) \, d\alpha(\bar{\nu})$, where 
$\alpha$ is a probability measure on the set of ergodic $F$-invariant $\mathbb{P}$-Markovian measures on $\Omega\times X$. Since there does not exist a mesurable family of probability measures $\{\mu_t\}_{t\in T}$ such that $f_{\omega_0}\mu_{\omega_0}=\mu_{\omega_1}$ for $\mathbb{P}$-a.e.~$\omega=(\omega_i)_{i\geq 0}\in \Omega$, Theorem~\ref{thm:Markovian-inv-principle} implies that $\lambda(\bar{\nu})<0$ for all $\bar{\nu}$ in the support of $\alpha$ and consequently, $\lambda(\bar\mu)<0$. This implies that $f$ is $\mathbb{P}$-mostly contracting, proving the former half of the statement.

Furthermore, if $T$ is a finite set, then the standing assumption in Section~\ref{sec:iid-representation} is satisfied, as indicated in the previous section.
This concludes the proof of the latter statement of the corollary by applying Theorem~\ref{cor:Markovian-qc}.

\subsection{Proof of Corollary~\ref{cor:interval}} 
Since an interval $C^1$ diffeomorphisms onto its image can be extended to $\mathbb{S}^1$, we can extend $f: T\times \mathbb{I} \to \mathbb{I}$ to a random map $\tilde{f}: T \times \mathbb{S}^1 \to \mathbb{S}^1$, such that for any point $x \in \mathbb{S}^1 \setminus \mathbb{I}$ and every $\omega \in \Omega$, there exists $n_0 \in \mathbb{N}$ such that $\tilde{f}^n_\omega(x) \in \mathbb{I}$ for all $n \geq n_0$.

Let \( F: \Omega \times \mathbb{I} \to \Omega \times \mathbb{I} \) and \( \tilde{F}: \Omega \times \mathbb{S}^1 \to \Omega \times \mathbb{S}^1 \) be the skew-product maps associated with \( f \) and \( \tilde{f} \), respectively. Since \( T \) is finite and each \( f_t \) (and thus \( \tilde{f}_t \)) is \( C^1 \) on a compact manifold, the condition that \( \log^+ \|\tilde{f}'_{\omega_0}\|_\infty \) is \( \mathbb{P} \)-integrable is automatically satisfied.

We proceed by contradiction. Assume that \( f \) is not \( \mathbb{P} \)-mostly contracting. By definition, this implies that there exists at least one ergodic, \( F \)-invariant, \( \mathbb{P} \)-Markovian measure \( \bar{\mu} \) on \( \Omega \times \mathbb{I} \) such that \( \lambda(\bar{\mu}) \ge 0 \).
Any such \( F \)-invariant measure \( \bar{\mu} \) on \( \Omega \times \mathbb{I} \) can be extended to an \( \tilde{F} \)-invariant measure on \( \Omega \times \mathbb{S}^1 \) by defining \( \bar{\mu}(\Omega \times (\mathbb{S}^1 \setminus \mathbb{I})) = 0 \).
Hence, since the Lyapunov exponent \( \lambda(\bar{\mu}) \) remains the same whether calculated with respect to \( F \) or \( \tilde{F} \), we get that there exists an ergodic, \( \tilde{F} \)-invariant, \( \mathbb{P} \)-Markovian measure \( \bar{\mu} \) on \( \Omega \times \mathbb{S}^1 \) such that \( \lambda(\bar{\mu}) \ge 0 \).
According to Theorem~\ref{thm:Markovian-inv-principle}, $\bar\mu$  must satisfy the invariance condition $f_{\omega_0}\mu_{\omega_0}=\mu_{\omega_1}$  for $\mathbb{P}$-a.e.~$\omega=(\omega_i)_{i\geq 0}\in \Omega$, where $d\bar\mu=\mu_{\omega_0}\, d\mathbb{P}(\omega)$. 
Since $\mathbb{P}$ is a Markov measure with the transition probability $Q(t,A)$ and the initial distribution $p$, this invariance condition is 
  equivalently to the following:   
  \begin{equation} \label{eq:52}
 f_{t}\mu_{t} = \mu_{a} \quad \text{for $p$-a.e.~$t\in T$ and $Q(t,\cdot)$-a.e.~$a\in T$.}
 \end{equation}
Since \( T \) is finite and \( Q(t,A) \) is irreducible, the unique stationary measure \( p \) satisfies \( p(\{t\}) > 0 \) for all \( t \in T \). Furthermore, ``\( Q(t, \cdot) \)-a.e.~\( a \in T\)'' means ``for all \( a \in T \) such that \( Q(t, \{a\}) > 0 \)''. Thus,~\eqref{eq:52} becomes 
\begin{equation} \label{eq:53}
 f_{t}\mu_{t} = \mu_{a} \quad \text{for all $t,a\in T$ such that  $Q(t,\{a\})>0$.}
 \end{equation}

Now we prove that there exists a finite invariant orbit under the Markovian random dynamics of $f$, which is impossible by assumption.  To see this, observe that~\eqref{eq:53} implies \( \supp {f}_t \mu_t = \supp\mu_a \) for all $t,a\in T$ such that $Q(t,\{a\})>0$. As \( {f}_t \) is continuous and \( S_t=\supp \mu_t \) is compact, we get that 
\begin{equation*} \label{eq:support-invariance-proof-latex}
 S_a=\supp \mu_a= \supp {f}_t \mu_t =  {f}_t(S_t) \quad \text{for all } t,a \in T  \text{ such that } Q(t, \{a\}) > 0.
\end{equation*}
Let \( x_t = \inf S_t \), \( y_t = \sup S_t \) and $A_t=\{x_t,y_t\}$. Since \( S_t \subset \mathbb{I} \) is compact, \( A_t \subset S_t \). As \( f_t \) is a diffeomorphism on \( \mathbb{I} \), it is strictly monotone on the interval \( [\inf S_t, \sup S_t] \). The relation \( {f}_t(S_t) = S_a \) implies that \( {f}_t(A_t) = A_a \). Then \( \{A_t\}_{t \in T} \) is a collection of two points in \( \mathbb{I} \)  such that
\[
f_t(A_t) = A_a \quad \text{for all } t, a \in T  \text{ such that } Q(t, \{a\}) > 0.
\]
The existence of such a collection contradicts the hypothesis of the corollary. Therefore,  \( f \) is  \( \mathbb{P} \)-mostly contracting. The final statement regarding the conclusions of Theorem~\ref{cor:Markovian-qc} follows directly from the fact that \( f \) is \( \mathbb{P} \)-mostly contracting together with the hypothesis that \( T \) is finite.

\section*{Acknowledgement}
P.~G.~Barrientos was supported by grant PID2020-113052GB-I00 funded by MCIN, PQ 305352/2020-2 (CNPq), and JCNE E-26/201.305/2022 (FAPERJ). 
F.~Nakamura was supported by JSPS KAKENHI Grant Numbers 23K25785 and 24K16942.
Y.~Nakano was supported by JSPS KAKENHI Grant  Numbers 21K03332, 22K03342 and 23K03188.
H.~Toyokawa was supported by JSPS KAKENHI Grant Number 24K16943.


\def\cprime{$'$}

\end{document}